\documentclass{amsart}

\usepackage{amssymb, amscd, latexsym, graphicx, psfrag}
\usepackage[all]{xy}

\newtheorem{dummy}{dummy}[section]
\newtheorem{lemma}[dummy]{Lemma}
\newtheorem{theorem}[dummy]{Theorem}

\newtheorem{corollary}[dummy]{Corollary}
\newtheorem{proposition}[dummy]{Proposition}
\theoremstyle{definition}
\newtheorem{definition}[dummy]{Definition}
\newtheorem{example}[dummy]{Example}
\newtheorem{remark}[dummy]{Remark}
\newtheorem{remarks}[dummy]{Remarks}

\newcommand{\bC}{\mathbb{C}}

\newcommand{\bP}{\mathbb{P}}
\newcommand{\bQ}{\mathbb{Q}}
\newcommand{\bR}{\mathbb{R}}
\newcommand{\bZ}{\mathbb{Z}}

\newcommand{\bfP}{\mathbf{P}}
\newcommand{\bSi}{\mathbf{\Si} }

\newcommand{\cA}{\mathcal{A}}
\newcommand{\cB}{\mathcal{B}}

\newcommand{\cD}{\mathcal{D}}

\newcommand{\cF}{\mathcal{F}}
\newcommand{\cG}{\mathcal{G}}
\newcommand{\cH}{\mathcal{H}}
\newcommand{\cK}{\mathcal{K}}
\newcommand{\cL}{\mathcal{L}}
\newcommand{\cP}{\mathcal{P}}
\newcommand{\cQ}{\mathcal{Q}}
\newcommand{\cO}{\mathcal{O}}

\newcommand{\cT}{\mathcal{T}}

\newcommand{\cV}{\mathcal{V}}
\newcommand{\cX}{\mathcal{X}}

\newcommand{\Hom}{\mathrm{Hom}}
\newcommand{\dghom}{\mathit{hom}}           

\newcommand{\Ker}{\mathrm{Ker}}

\newcommand{\gr}{\mathrm{gr}}
\newcommand{\qgr}{\mathrm{qgr}}

\newcommand{\tors}{\mathrm{tors}}
\newcommand{\rig}{ {\mathrm{rig}} }
\newcommand{\pr}{\mathrm{pr}}

\newcommand{\fofr}{\mathfrak{or}}

\newcommand{\Ga}{ {\Gamma} }
\newcommand{\si}{ {\sigma} }
\newcommand{\Si}{ {\Sigma} }

\newcommand{\tC}{\widetilde{C}}
\newcommand{\tH}{\widetilde{H}}
\newcommand{\tL}{\widetilde{L}}
\newcommand{\tM}{\widetilde{M}}
\newcommand{\tN}{\widetilde{N}}
\newcommand{\tD}{\widetilde{D}}
\newcommand{\tT}{\widetilde{T}}

\newcommand{\ta }{\widetilde{a}}
\newcommand{\tb} {\widetilde{b}}

\newcommand{\tf }{\widetilde{f}}

\newcommand{\tj }{\widetilde{j}}
\newcommand{\tit}{\widetilde{t}}

\newcommand{\tu }{\widetilde{u}}

\newcommand{\tsi}{ {\widetilde{\si}} }
\newcommand{\tSi}{ {\widetilde{\Si}} }
\newcommand{\tchi}{ {\widetilde{\chi} }}

\newcommand{\tcF}{ \widetilde{\cF} }
\newcommand{\tcV}{ \widetilde{\cV} }

\newcommand{\baN}{\bar{N}}
\newcommand{\bab}{\bar{b}}

\newcommand{\vc}{ {\vec{c}} }

\newcommand{\LS}{ {\Lambda_{\Si}} }
\newcommand{\LbS}{ {\Lambda_{\bSi}} }

\newcommand{\uchi}{\underline{\chi}}
\newcommand{\Spec}{\mathrm{Spec}}
\newcommand{\Ext}{\mathrm{Ext}}
\newcommand{\Tor}{\mathrm{Tor}}
\newcommand{\frakm}{\mathfrak{m}}

\newcommand{\Ntor}{N_{\mathrm{tor}} }

\newcommand{\Sh}{\mathit{Sh}}       
\newcommand{\naive}{\mathit{naive}} 
\newcommand{\Tr}{\mathit{Tr}}           
\newcommand{\Perf}{\cP\mathrm{erf}} 
\newcommand{\Ob}{\mathit{Ob}}       
\newcommand{\Shard}{\mathrm{Shard}}

\newcommand{\finfib}{\mathit{fin}}      

\renewcommand{\SS}{\mathit{SS}}     
\newcommand{\ltr}{\langle \Theta \rangle}
\newcommand{\ltrp}{\langle\Theta'\rangle}

\begin{document}
\title[Coherent-Constructible Correspondence for Toric DM Stacks]{The 
Coherent-Constructible Correspondence for Toric Deligne-Mumford Stacks}

\author{Bohan Fang}
\address{Bohan Fang, Department of Mathematics, Northwestern University,
2033 Sheridan Road, Evanston, IL  60208}
\email{b-fang@math.northwestern.edu}

\author{Chiu-Chu Melissa Liu}
\address{Chiu-Chu Melissa Liu, Department of Mathematics, Columbia University,
2990 Broadway, New York, NY 10027} \email{ccliu@math.columbia.edu}

\author{David Treumann}
\address{David Treumann, Department of Mathematics, Northwestern University,
2033 Sheridan Road, Evanston, IL 60208}
\email{treumann@math.northwestern.edu}

\author{Eric Zaslow}
\address{Eric Zaslow, Department of Mathematics, Northwestern University,
2033 Sheridan Road, Evanston, IL  60208}
\email{zaslow@math.northwestern.edu}

\begin{abstract}
We extend our previous work \cite{more} on coherent-constructible correspondence for
toric varieties to include toric Deligne-Mumford (DM) stacks.  
Following Borisov-Chen-Smith \cite{BCS},
a toric DM stack $\cX_\bSi$ is described by a ``stacky fan'' $\bSi=(N,\Si,\beta)$,
where $N$ is a finitely generated abelian group and $\Si$ is a simplicial fan
in  $N_\bR=N\otimes_{\bZ}\bR$.  From $\bSi$ we define a
conical Lagrangian $\Lambda_\bSi$ inside the cotangent $T^*M_\bR$
of the dual vector space $M_\bR$ of $N_\bR$, such that torus-equivariant, coherent sheaves on
$\cX_\bSi$ are equivalent to constructible sheaves on $M_\bR$
with singular support in $\LbS$. 

The microlocalization theorem of Nadler and the last author \cite{NZ,N} relates
constructible sheaves on $M_\bR$ to a Fukaya category on the cotangent $T^*M_\bR$, 
giving a version of homological mirror symmetry for toric DM stacks.
\end{abstract}
\maketitle

{\small \tableofcontents}

\section{Introduction}
\label{sec:introduction}

In \cite{more}, the familiar assignment in toric geometry
which associates polytopes to line bundles on toric varieties was
extended to an equivalence of categories between equivariant coherent sheaves
on a toric $n$-fold and constructible sheaves on $\bR^n.$
The equivalence -- the {\em coherent-constructible correspondence} (CCC),
is therefore a ``categorification'' of
Morelli's description of the K-theory of toric varieties in
terms of a polytope algebra \cite{M}.
Combining the CCC with the microlocalization theorem of Nadler and
the last author \cite{NZ, N}, one obtains an equivariant
version of homological mirror symmetry (HMS) which is compatible
with T-duality \cite{hms}.

In this paper, we extend CCC to toric Deligne-Mumford (DM) stacks,
and derive the attendant version of HMS.

In Section \ref{sec:ccc-varieties}
below, we recall for the reader the CCC and HMS for toric varieties.  We discuss the extension to
toric DM stacks in Section \ref{sec:ccc-orbifolds}, and illustrate it in
the simple example of a weighted projective plane
in Section \ref{sec:WPtwo}.  An outline of the paper follows.

In this paper, $\Si$ is always a {\em finite} fan, i.e., 
$\Si$ consists of finitely many cones. Therefore all the toric varieties
and toric DM stacks in this paper are {\em of  finite type}.

\subsection{CCC and HMS for toric varieties}
\label{sec:ccc-varieties}

Let $N \cong \bZ^n$ be a rank $n$ lattice, and let $\Si$ be a complete
fan in $N_\bR = N\otimes_\bZ\bR$. Let $X_\Si$ be the toric
variety defined by $\Si$, and let $T\cong (\bC^*)^n$ be the torus acting on $X_\Si$.
In \cite{more}, we defined a conical Lagrangian $\LS$ in the cotangent bundle
$T^*M_\bR$ of $M_\bR$, where $M_\bR$ is the dual real vector space of $N_\bR$,
and established a quasi-equivalence of triangulated dg-categories:\footnote{
The results in \cite{more} is valid over any commutative neotherian base ring.
The nonequivariant CCC has been studied by Bondal  \cite{Bondal},
and more recently by the third author \cite{Tr}.}
\begin{equation}\label{eqn:kappa}
\kappa: \Perf_T(X_\Si) \to Sh_{cc}(M_\bR;\LS)
\end{equation}
where $\Perf_T(X_\Si)$ is the category of $T$-equivariant perfect complexes on $X_\Si$,
and $Sh_{cc}(M_\bR;\LS)$ is the category of complexes of sheaves on $M_\bR$ with bounded,
constructible, compactly supported cohomology,  with singular support in $\LS$. Moreover,
\eqref{eqn:kappa} is a quasi-equivalence of {\em monoidal} dg categories,
with respect to the tensor product on $\Perf_T(X_\Si)$ and the convolution
product on $Sh_{cc}(M_\bR;\LS)$.
Combining \eqref{eqn:kappa} with the microlocalization theorem
\cite{NZ, N}, we have the following commutative diagram:
\begin{equation}\label{eqn:quasi}
\xymatrix{
\Perf_T(X_\Si) \ar[r]^{\kappa}  \ar[dr]_{\tau}
& Sh_{cc}(M_\bR; \LS) \ar[d]_\mu \\
& F(T^*M_\bR; \LS)
}
\end{equation}
where $F=TrFuk$ is the triangulated envelope of the Fukaya category. The objects in
$Fuk(T^*M_\bR;\LS)$ are Lagrangian branes which are bounded in the $M_\bR$ direction,
with boundary at infinity contained in $\LS$.
Here $\kappa$ is the CCC, and $\mu$ is referred to as the  {\em microlocalization} functor, which is a quasi-equivalence
of triangulated $A_\infty$-categories by the microlocalization theorem
\cite{NZ, N}, and $\tau=\mu\circ \kappa$ is referred to as the {\em T-duality} functor
because it is compatible with the T-duality \cite{hms}.

When $X_\Si$ is a smooth complete toric variety (so that $X_\Si$
can be viewed as a compact complex manifold), taking the cohomology
of \eqref{eqn:quasi} yields equivalences of three tensor\footnote{The
product on $DF(T^*M_\bR;\LS)$ is defined in \cite[Section 3.4]{hms}.}  
triangulated categories: 
\begin{equation}\label{eqn:H}
\xymatrix{
D_T(X_\Sigma) \ar[r]_\cong^{ H(\kappa) } \ar[dr]^\cong_{ H(\tau)}
& D_{cc}(M_\bR; \LS) \ar[d]_\cong^{ H(\mu)} \\
& DF(T^*M_\bR; \LS)
}
\end{equation}
where $D_T(X_\Si)= DCoh_T(X_\Si)$ is the bounded derived category of $T$-equivariant
coherent sheaves on $X_\Si$, and  $D_{cc}(M_\bR;\LS)= DSh_{cc}(M_\bR;\LS)$ is
the derived category of $Sh_{cc}(M_\bR;\LS)$. The equivalence $H(\tau)$ can be viewed as
a $T$-equivariant version of {\em  homological mirror symmetry} (HMS).

\subsection{CCC and HMS for toric DM stacks}
\label{sec:ccc-orbifolds}

Now let $N$ be a finitely generated abelian group, and let $\Si$ be
a complete {\em simplicial} fan in
$N_\bR=N\otimes_\bZ\bR$. Let $\bSi=(N,\Si,\beta)$ be a stacky fan in the sense of Borisov-Chen-Smith \cite{BCS}. 
Then $\bSi$ defines a complete toric DM stack $\cX_{\bSi}$, and
the complete toric variety $X_\Si$ is the coarse moduli space of $\cX_{\bSi}$.
A toric DM stack $\cX_{\bSi}$ is a toric orbifold (i.e. a toric
DM stack with trivial generic stabilizer) if $N$ is torsion free.
Given any stacky fan $\bSi$ there is a rigidification
$\bSi^\rig$ which defines an toric {\em orbifold} $\cX_{\bSi^\rig}$,
known as the rigidification of the toric DM stack $\cX_{\bSi}$ 
(see Section \ref{sec:rigidification}).

A toric DM stack $\cX_{\bSi}$ contains a DM torus $\cT= T\times \cB G$
as a dense open subset, where $T\cong (\bC^*)^{\dim_\bR N_\bR}$ is a torus and 
$G$ is the generic stabilizer. 
We define a conical Lagrangian $\LbS$ of $T^*M_\bR$
($\LS$ is a subset of $\LbS$) and establish
a quasi-equivalence of monoidal triangulated dg-categories
\begin{equation}\label{eqn:kappaDM}
\kappa: \Perf_\cT(\cX_\bSi) \to Sh_{cc}(M_\bR;\LbS)
\end{equation}
which is the stacky version of \eqref{eqn:kappa}.
Again, combining \eqref{eqn:kappaDM} with
the microlocalization theorem \cite{NZ, N}, we obtain the following
stacky version of \eqref{eqn:quasi}:
\begin{equation}\label{eqn:quasiDM}
\xymatrix{
\Perf_\cT(\cX_\bSi) \ar[r]^{\kappa}  \ar[dr]_\tau  &
Sh_{cc}(M_\bR; \LbS) \ar[d]_\mu \\
&F(T^*M_\bR; \LbS)
}
\end{equation}
where $\kappa$ is the CCC for toric DM stacks,   
$\mu$ is the microlocalization functor, and
$\tau = \mu\circ\kappa$ is the $T$-duality functor.
We remark that (i) $\Perf_{\cT}(\cX_\bSi)$ 
depends only on the rigidification $\cX_{\bSi^\rig}$
of $\cX_{\bSi}$, which is not the case for
the nonequivariant category $\Perf(\cX_\bSi)$ (see
Proposition \ref{rigid} for precise statements),
and (ii) $\LbS = \Lambda_{\bSi^\rig}$.


Taking cohomology of \eqref{eqn:quasiDM} 
yields equivalences of tensor triangulated categories,
exactly as in \eqref{eqn:H}, and $H(\tau)$ can be viewed as an equivariant
version of HMS for toric DM stacks.

The traditional version for HMS of toric DM stacks 
relates the bounded derived category of coherent
sheaves on a toric DM stack to the Fukaya-Seidel
category of the Landau-Ginzburg mirror of the toric DM stack. 
This version has been proved for weighted projective planes by 
Auroux-Katzarkov-Orlov \cite{AKO1}, for toric orbifolds 
of toric del Pezzo surfaces by Ueda-Yamazaki \cite{UeYa},
and more recently, for toric orbifolds of projective spaces
(of any dimension) by Futaki-Ueda \cite{FuUe}. Of course, there are 
many other works on homological mirror symmetry for
\emph{non-stacky} toric varieties. See, e.g., \cite[Section 1.3]{hms} for
a review of these results and further references.

\subsection{Example: the weighted projective plane $\bP(1,1,2)$}
\label{sec:WPtwo}

Let $N=\bZ^2$. The weighted projective plane $\cX_\bSi=\bP(1,1,2)$ is defined from the
stacky fan $\bSi = (N,\Si,\beta)$ in Figure 1.

\begin{figure}[h]
\begin{center}
\psfrag{S}{$\Si$}
\psfrag{BBBBBBBB}{\small $\quad \beta(\left[\begin{array}{c}n_1\\ n_2 \\n_3\end{array}\right])
= \left[\begin{array}{ccc} 1 & 0 & -1\\ 0 & 1 &  2 \end{array}\right]
\left[\begin{array}{c}n_1\\ n_2 \\n_3\end{array}\right]$ }
\includegraphics[scale=0.5]{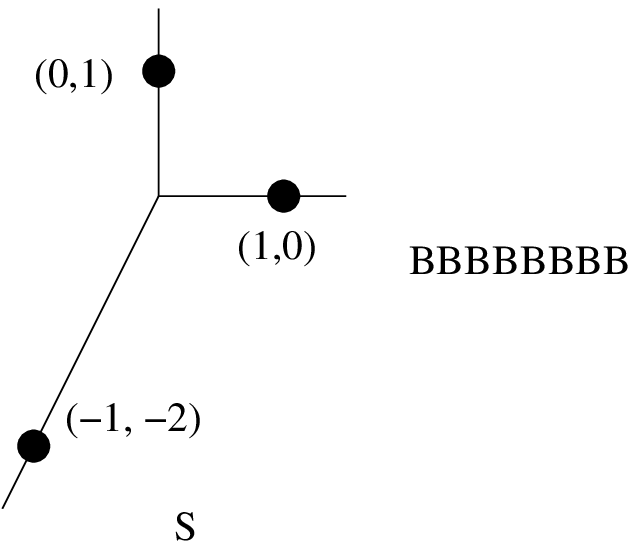}
\end{center}
\caption{Stacky Fan for $\bP(1,1,2)$}
\end{figure}

We first describe the derived category of $\bP(1,1,2)$, following
\cite[Section 2]{AKO1}. Define a graded polynomial algebra $S=\bC[x_0,x_1,x_2]$
graded by setting the degrees of $x_0,x_1,x_2$ to be $1,1,2$
respectively. The derived category of $\bP(1,1,2)$ is the category
$\qgr(S)=\gr(S)/\tors(S)$, where $\gr(S)$ is the category of finitely generated graded 
right $S$-modules, and $\tors(S)$ is the full subcategory of
$\gr(S)$ consisting of $S$-modules with finite dimensions
over $\bC$. Under this identification, the line bundle $\cO(l)$ over
$\bP(1,1,2)$ is $\widetilde {S(l)}$. Here $S(l)$ is the module $S$
in $\gr(S)$ shifted in $l$ degrees, and passing to
$\qgr(S)$, we denote it by $\widetilde{S(l)}$.

Decompose $S=\oplus_{i=0}^\infty S_i$ via degrees. The weighted
projective plane $\bP(1,1,2)$, as a stack, has a full strong exceptional
collection $\{\cO,\cO(1),\cO(2),\cO(3)\}$, with the morphism
$$\Ext^i(\cO(k),\cO(l))=\begin{cases}S_{l-k}, & \text{$l\ge k$, $i=0$},\\ 0,& \text{otherwise}, \end{cases}$$
and the obvious composition in the polynomial algebra $S$.

If we consider the singular toric variety (defined as GIT quotient)
given by the fan $\Si$ in Figure 1, we are dealing with the singular variety
$X_\Si= \bfP(1,1,2)$, not the stack. The result of \cite{more} says the
category of equivariant coherent sheaves on $\bfP(1,1,2)$ is
quasi-equivalent to a subcategory of constructible sheaves on the
plane $\bR^2$. 
\begin{figure}[h]
\begin{center}
\psfrag{LS}{$\LS/M$}
\psfrag{LbS}{$\LbS/M$}
\includegraphics[scale=0.6]{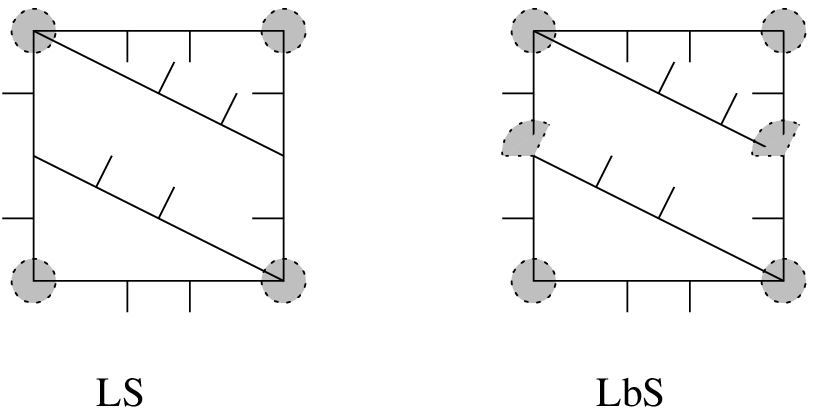}
\end{center}
\caption{Microlocal Supports for Constructible Sheaves: $\bfP(1,1,2)$ (left) and $\bP(1,1,2)$ (right)}
\end{figure}
More precisely, the subcategory of constructible
sheaves, denoted by $Sh_{cc}(\bR^2, \LS)$, is generated by
the pushing-forward with compact support of constant sheaves on the
triangles with vertices $(a,b)$, $(a,b+h)$, $(a+2h, b)$ with $a,b,h
\in\bZ$, $h\ge0$. Comparing with the line bundles on the stack
$\bP(1,1,2)$, we are only recovering the (all possible equivariant
versions of) sheaves $\cO, \cO(2),\cO(4),\dots$. However, if we
refine the lattice points, and allow $h$ to take values in half
integers, we have the whole derived category on $\bP(1,1,2)$.

\subsection{Outline}
\label{sec:outline} In Section \ref{sec:notation}, we set the
notation and describe the categories of sheaves we will employ.  In
Section \ref{sec:toricDM} we review the definition
of stacky fans and the construction of toric DM stacks, following
\cite{BCS} and \cite{FMN}. In Section \ref{sec:twisted},
we describe the relation between equivariant line bundles and (twisted) polytopes. 
In Section \ref{sec:Theta-Theta}, we construct the basic equivalence between equivariant ${\mathcal
O}$-modules on toric charts and $\bC$-modules on shifted dual cones.
In Section \ref{sec:reformulate}, we give intrinsic characterizations
of the categories introduced in Section \ref{sec:Theta-Theta}.
In Section \ref{sec:Perf-cc}, we derive our main result,
the coherent-constructible correspondence (CCC), which is
an equivalence between equivariant perfect complexes on the
toric DM stack and compactly supported constructible sheaves on 
$M_\bR$ with singular support in a conical Lagrangian determined
by the stacky fan. In Section \ref{sec:HMS}, we combine CCC and microloalization
to obtain equivariant HMS for toric DM stacks.

\subsection*{Acknowledgments}
We thank Hsian-Hua Tseng for his helpful suggestions.
We thank Johan de Jong, Fabio Nironi, and Jason Starr for
helpful conversations.
The work of EZ is supported in part by NSF/DMS-0707064.

\section{Notation and Conventions}
\label{sec:notation}

\subsection{Categories}
\label{sec:categories}

We will use the language of dg and $A_\infty$ categories throughout.
If $C$ is a dg or $A_\infty$ category, then $\dghom(x,y)$ denotes
the chain complex of homomorphisms between objects $x$ and $y$ of
$C$.  We will use $\Hom(x,y)$ to denote hom sets in non-dg or
non-$A_\infty$ settings.  We will regard the differentials in all
chain complexes as having degree $+1$, i.e. $d:K^i \to K^{i+1}$.  If
$K$ is a chain complex (of vector spaces or sheaves, usually) then
$h^i(K)$ will denote its $i$th cohomology object.

If $C$ is a dg (or $A_\infty$) category, then $\Tr(C)$ denotes  the
triangulated dg (or $A_\infty$) category generated by $C$, and
$D(C)$ denotes the cohomology category $H(Tr(C)).$\footnote{Here is one
construction of the triangulated envelope. The Yoneda
embedding $\mathcal Y:C\rightarrow mod(C)$ maps an object $L\in C$
to the  $A_\infty$ right $C$-module $hom_{C}(-, L)$. The  functor
$\mathcal Y$ is a quasi-embedding of $C$ into the triangulated
category $mod(C)$. Let $Tr(C)$ denote the category of twisted
complexes of representable modules in $mod(C).$  Then $Tr(C)$ is a
triangulated envelope of $C$. } The triangulated category
$H(Tr(C))$ is sometimes called the \emph{derived category} of $C$.

\subsection{Constructible and microlocal geometry}
\label{sec:constructible}

We refer to \cite{KS} for the microlocal theory of sheaves. If $X$
is a topological space we let $\Sh(X)$ denote the dg category  of
bounded chain complexes of sheaves of $\bC$-vector spaces on $X$,
localized with respect to acyclic complexes (see \cite{Dr} for localizations
of dg categories).  If $X$ is a
real-analytic manifold, $\Sh_c(X)$ denotes the full subcategory of
$\Sh(X)$ of objects whose cohomology sheaves are constructible with
respect to a real-analytic stratification of $X$. Denote by
$\Sh_{cc}(X) \subset \Sh_c(X)$ the
full subcategory of objects which have compact support. We use
$D_c(X)$ and $D_{cc}(X)$ to denote the derived categories $D(Sh_c(X))$
and $D(Sh_{cc}(X))$ respectively.

The \emph{standard constructible sheaf} on the submanifold $i_{Y}:
Y\hookrightarrow X$  is defined as the push-forward of the constant
sheaf on $Y$, i.e. $i_{Y*} \bC_Y,$ as an object in $Sh_c(X)$. The
Verdier duality functor $\cD: Sh_c^\circ(X)\to \Sh_c(X)$ takes
$i_{Y*}\bC_Y$ to the \emph{costandard constructible sheaf} on $X$.
We know $\cD(i_{Y*}\bC_Y)=i_{Y!}\cD(\bC_Y)=i_{Y!} \omega_Y$. Here
$\omega_Y=\cD(\bC_Y)=\fofr_Y[\dim Y]$, where $\fofr_Y$ is the
orientation sheaf of $Y$ (with respect to the base ring $\bC$).

We denote the singular support of a complex of sheaves $F$ by
$\SS(F) \subset T^*X$.   If $X$ is a real-analytic manifold and
$\Lambda \subset T^*X$ is an $\bR_{\geq 0}$-invariant Lagrangian
subvariety, then $\Sh_c(X;\Lambda)$ (resp. $\Sh_{cc}(X;\Lambda)$)
denotes the full subcategory of $\Sh_c(X)$ (resp. $\Sh_{cc}(X)$)
whose objects have singular support in $\Lambda$.

\subsection{Coherent and quasicoherent sheaves}
\label{sec:coherent}

All schemes and stacks that appear will be over $\bC$.

\subsubsection{Sheaves on a scheme}
If $X$ is a scheme, then we let $\cQ(X)^\naive$ denote the dg category of bounded complexes of quasicoherent sheaves on $X$,
and we let $\cQ(X)$ denote the localization of this category with respect to acyclic complexes. 
If $G$ is an algebraic group acting on $X$, we let $\cQ_G(X)^\naive$ denote the dg category of complexes of
$G$-equivariant quasicoherent sheaves.  We let $\cQ_G(X)$ denote the localization of this category with respect to acyclic
complexes.   We use $\Perf(X) \subset \cQ(X)$ and $\Perf_G(X)\subset \cQ_G(X)$ to denote the full dg subcategories
consisting of \emph{perfect} objects---that is, objects which are quasi-isomorphic to bounded complexes of vector bundles.
If $u:X \to Y$ is a morphism of schemes, we have natural dg functors $u_*:\cQ(X) \to \cQ(Y)$ and $u^*:\cQ(Y) \to \cQ(X)$.
Note that the functor $u^*$ carries $\Perf(Y)$ to $\Perf(X)$. Suppose $G$ and $H$ are algebraic groups, $X$ is a scheme
with a $G$-action, and $Y$ is a scheme with an $H$-action.  If a morphism $u:X \to Y$ is equivariant with respect to
a homomorphism of groups $\phi:G \to H$, then we will often abuse notation and write $u_*$ and $u^*$ for the equivariant pushforward and pullback functors $u_*:\cQ_G(X) \to \cQ_H(Y)$ and $u^*:\cQ_H(Y) \to \cQ_G(X)$.

\subsubsection{Sheaves on a DM stack} \label{sec:sheaves-DM} 
We refer to \cite[Definition 7.18]{Vi} for the definitions of
quasicoherent sheaves, coherent sheaves, and vector bundles on a DM
stack. If $\cX$ is a DM stack, then we let $\cQ(\cX)^\naive$ denote the dg category of bounded
complexes of quasicoherent sheaves on $\cX$, and let $\cQ(\cX)$ denote the localization of this category
with respect to acyclic complexes.
We use $\Perf(\cX) \subset \cQ(\cX)$ to denote the full dg subcategories
consisting of \emph{perfect} objects---that is, objects which are quasi-isomorphic to bounded complexes of vector bundles.

\subsubsection{Sheaves on a global  quotient}\label{sec:global-quotient}
We now spell out the above definitions when $\cX$ is a global quotient.
Let $G$ be an algebraic group acting on a scheme $U$, such that the stabilizers of the geometric
points of $U$ are finite and reduced. By \cite[Example 7.17]{Vi}, the quotient stack
$\cX =[U/G]$ is a DM stack.  By \cite[Example 7.21]{Vi},
the category of coherent sheaves on $\cX$ is equivalent to the category of $G$-equivariant coherent sheaves on $U$.
Similarly, the category of quasicoherent sheaves on $\cX$ is equivalent to the category of
$G$-equivariant quasicoherent sheaves on $U$. Therefore,
$$
\cQ(\cX)^\naive = \cQ_G(U)^{\naive}, \quad \cQ(\cX) = \cQ_G(U),\quad \Perf(\cX)=\Perf_G(U).
$$

Now suppose that, in addition, $G$ is an {\em abelian} group, and
there is an abelian group $\tH$ acting on $U$, such that
the $G$-action on $U$ factors through a group homomorphism
$\phi:G\to \tH$.  Then $\cH=[\tH/G]$ is a Picard stack
acting on $\cX=[U/G]$ in the sense of \cite{FMN}. We define the category of $\cH$-equivariant coherent (resp. quasicoherent) 
sheaves on $\cX$ to be equivalent to the category of $\tH$-equivariant coherent (resp. quasicoherent) 
sheaves on $U$: 
$$
\cQ_{\cH}(\cX) = \cQ_{\tH}(U),\quad \Perf_{\cH}(\cX)=\Perf_{\tilde{H}}(U).
$$

\section{Preliminaries on Toric DM Stacks}
\label{sec:toricDM}

In \cite{BCS}, Borisov, Chen, and Smith defined toric Deligne-Mumford (DM) stacks
in terms of stacky fans. Toric DM stacks are smooth DM
stacks, and their coarse moduli spaces are simplicial toric varieties.
A toric DM stack is called a {\em reduced toric DM stack} (in \cite{BCS})
or a {\em toric orbifold} (in \cite{FMN}) if its generic stabilizer is trivial.
Later, more geometric definitions of toric orbifolds and
toric DM stacks are given by Iwanari \cite{Iw1, Iw2}
and by Fantechi-Mann-Nironi \cite{FMN}, respectively.

\subsection{Stacky fans} \label{sec:stacky}

In this subsection, we recall the definition of stacky
fans. Let $N$ be a finitely generated abelian group, and
let $N_\bR=N\otimes_\bZ\bR$.
We have a short exact sequence of abelian groups:
$$
1\to \Ntor\to N\to \baN=N/\Ntor \to 1,
$$
where $\Ntor$ is the subgroup of torsion elements in 
$N$. Then $\Ntor$ is a finite abelian group, and $\baN\cong \bZ^n$,
where $n=\dim_\bR N_\bR$. 
The natural projection $N\to \baN$ is denoted by $b\mapsto \bab$.

Let $\Si$ be a simplicial fan in $N_\bR$ (see \cite{Fu}), and let 
$\Si(1)=\{\rho_1,\ldots,\rho_r\}$ be
the set of 1-dimensional cones in the fan $\Si$.
We assume that $\rho_1,\ldots, \rho_r$ span $N_\bR$, and 
fix $b_i\in N$ such that $\rho_i=\bR_{\geq 0}\bab_i$. 
A \emph{stacky fan} $\bSi$ is defined as the data
$(N,\Si,\beta)$, where  
$\beta:\tN:=\oplus_{i=1}^r \bZ \tb_i\cong \bZ^r \to N$
is a group homomorphism defined by  $\tb_i\mapsto b_i$.
By assumption, the cokernel of $\beta$ is finite.

We introduce some notation. 
\begin{enumerate}
\item $M=\Hom(N,\bZ)=\Hom(\baN,\bZ)\cong (\bZ^n)^*$. 
\item $\tM=\Hom(\tN,\bZ)\cong (\bZ^r)^*$. 
\item Let $\Si(d)$ be the set of $d$-dimensional
cones in $\Si$. Given $\si\in \Si(d)$, let 
$N_\si\subset N$ be the subgroup generated by
$\{ b_i\mid \rho_i\subset \si\}$, 
and let $\baN_\si$ be the rank $d$ sublattice of $\baN$
generated by $\{\bab_i\mid \rho_i\subset \si\}$. Let
$M_\si=\Hom(\baN_\si,\bZ)$ be the dual lattice of $\baN_\si$.
\end{enumerate}

Given $\si\in \Si(d)$, the surjective group homomorphism $N_\si\to \baN_\si$ induces
an injective group homomorphism $\Hom(\baN_\si,\bZ)\to \Hom(N_\si,\bZ)$
which is indeed an isomorphism. So $\Hom(N_\si,\bZ)\cong M_\si\cong \bZ^d$.

\subsection{The Gale dual} \label{sec:gale}
The finite abelian group $N_{\mathrm{tor}}$ is of 
the form $\oplus_{j=1}^l\bZ_{a_j}$.
We choose a projective resolution of $N$:
$$
0\to \bZ^l\stackrel{Q}{\to }\bZ^{n+l}\to N\to 0.
$$
Choose a map $B:\tN\to \bZ^{n+l}$ lifting
$\beta:\tN \to N$. Let $\pr_1:\tN\oplus \bZ^l\to \tN$
and $\pr_2:\tN\oplus\bZ^l\to \bZ^l$ be projections
to the first and second factors, respectively. We have the
following commutative diagram:
$$
\xymatrix{
    & &
\tN\oplus \bZ^l \ar[r]^{\pr_1}\ar[d]_{B\oplus Q}\ar[dl]_{\pr_2} & \tN \ar[d]^\beta \ar[dl]^B   & \\
0 \ar[r] &  \bZ^l  \ar[r]^Q &  \bZ^{n+l} \ar[r] & N\ar[r] & 0
} 
$$

Define the dual group $DG(\beta)$ to be the the cokernel of 
$B^*\oplus Q^*:(\bZ^{n+l})^*\longrightarrow \tM\oplus (\bZ^l)^*$.
The Gale dual of the map $\beta:\tN\to N$ is $\beta^\vee:\tM\to DG(\beta)$. 

$$
\xymatrix{
& 0 & \\
& DG(\beta)\ar[u] & \\
& \tM\oplus (\bZ^l)^* \ar[u] & \tM \ar[l]_{\pr_1^*}\ar[ul]_{\beta^\vee} \\
\bZ^l \ar[ur]^{\pr_2^*} & (\bZ^{n+l})^* \ar[u]^{ B^*\oplus Q^* }\ar[ur]_{B^*}\ar[l]_{Q^*} & \\
}
$$

\subsection{Construction of the toric DM stack} \label{sec:construct}
We follow \cite[Section 3]{BCS}.
Applying $\Hom(-,\bC^*)$ to $\beta^\vee:\tM\to DG(\beta)$, one obtains
$$
\phi: G_{\bSi}:=\Hom(DG(\beta),\bC^*)\to \tT:=\Hom(\tM,\bC^*).
$$
Let $G=\Ker\phi$. Then $G\cong \prod_{j=1}^l \mu_{a_j}$, where $\mu_{a_j}\subset \bC^*$ is 
the group of $a_j$-th roots of unity, which is isomorphic to $\bZ_{a_j}$. Let $\cB G$ denote the quotient stack $[\{0\}/G]$.
The algebraic torus $\tT$ acts on $\bC^r$ by 
$$
(\tit_1,\ldots, \tit_r)\cdot(z_1,\ldots, z_r)=(\tit_1 z_1,\ldots,\tit_r z_r),\quad
(\tit_1,\ldots, \tit_r)\in \tT,\quad (z_1,\ldots, z_r)\in \bC^r. 
$$
Let $G_{\bSi}$ act on $\bC^r$ by $g\cdot z:=\phi(g)\cdot z$, where $g\in G_{\bSi}$, $z\in \bC^r$.
Let $\cO(\bC^r) = \bC[z_1,\ldots, z_r]$ be the coordinate ring of $\bC^r$. Let $I_\Si$ be the
ideal of $\cO(\bC^r)$ generated by
$$
\{ \prod_{\rho_i \not\subset \si} z_i : \si\in \Si\}
$$
and let $Z(I_\Si)$ be the closed subscheme of $\bC^r$ defined by $I_\Si$.
Then $U:=\bC^r-Z(I_\Si)$ is a quasi-affine variety over $\bC$. The toric DM stack
associated to the stacky fan $\bSi$ is defined to be the quotient stack
$$
\cX_{\bSi}:= [U/G_{\bSi}].
$$
It is a smooth DM stack whose generic stabilizer is $G$, 
and its coarse moduli space is the toric variety $X_\Si$ defined by 
the simplicial fan $\Si$. There is an open dense immersion
$$
\iota: \cT=[ \tT/G_{\bSi} ] \hookrightarrow \cX_{\bSi} = [U/G_{\bSi}],
$$
where $\cT\cong (\bC^*)^n\times \cB G$ is a DM torus.
The action of $\cT$ on itself extends to an action 
$a:\cT\times \cX_{\bSi}\to \cX_{\bSi}$. 

\subsection{Rigidification} \label{sec:rigidification}

We define the {\em rigidification} of $\bSi=(N,\Si,\beta)$ to be the
stacky fan  $\bSi^\rig:=(\bar{N}, \Si, \bar{\beta})$, where
$\bar{\beta}$ is the composition of $\beta:\tN\to N$ with the projection $N\to \baN$.  
Note that $M$, $\baN_\si$, and $M_\si$ defined in Section 
\ref{sec:stacky} depend only on $\bSi^\rig$.
The generic stabilizer of the toric DM stack $\cX_{\bSi^\rig}$
is trivial because $\baN\cong \bZ^n$ is torsion free. So
$\cX_{\bSi^\rig}$ is a toric orbifold.
There is a morphism of stacky fans $\bSi\to \bSi^\rig$
which induces a morphism of toric DM stacks
$r: \cX_{\bSi}\to \cX_{\bSi^\rig}$. The toric orbifold $\cX_{\bSi^\rig}$ is 
called the {\em rigidification} of the toric DM stack $\cX_{\bSi}$.
The morphism $r:\cX_\bSi\to \cX_{\bSi^\rig}$ makes $\cX_\bSi$ 
an abelian gerbe over $\cX_{\bSi^\rig}$.
 
$G_{\bSi^\rig}= G_{\bSi}/G$ is a subgroup 
of $\tT$. Let $T:=\tT/G_{\bSi^\rig}\cong (\bC^*)^n$.
There is an open dense immersion
$$
\iota^\rig: T=[ \tT/G_{\bSi^\rig}]\hookrightarrow \cX_{\bSi^\rig} = [U/G_{\bSi^\rig}].
$$

We have the following statements (cf.  Section \ref{sec:global-quotient}).
\begin{proposition}\label{rigid}
\begin{enumerate}
\item[(a)] (nonequivariant sheaves on toric DM stacks) 
The morphism $r:\cX_{\bSi}\to \cX_{\bSi^\rig}$
induces the following quasi-embeddings of dg categories:
\begin{eqnarray*}
&& r^*: \Perf(\cX_{\bSi^\rig})\cong \Perf_{G_{\bSi^\rig}}(U)\hookrightarrow
\Perf(\cX_{\bSi})\cong \Perf_{G_\bSi}(U),\\
&& r^*: \cQ(\cX_{\bSi^\rig})\cong \cQ_{G_{\bSi^\rig}}(U)\hookrightarrow
\cQ(\cX_{\bSi})\cong \cQ_{G_\bSi}(U).
\end{eqnarray*}
They are quasi-equivalences if and only if $\bSi=\bSi^\rig$.
\item[(b)] (equivariant sheaves on toric DM stacks) We have the following quasi-equivalences of dg categories:
\begin{eqnarray*}
\Perf_T(\cX_{\bSi^\rig})\cong \Perf_\cT(\cX_\bSi)\cong \Perf_{\tT}(U),\\
\cQ_T(\cX_{\bSi^\rig})\cong \cQ_\cT(\cX_\bSi)\cong \cQ_{\tT}(U).
\end{eqnarray*}

\item[(c)] (forgetting the equivariant structure) The forgetful functors
$$
\Perf_{\cT}(\cX_{\bSi})\to \Perf(\cX_{\bSi}),\quad
\cQ_{\cT}(\cX_{\bSi})\to \cQ(\cX_{\bSi})
$$
are essentially surjective if and only if 
$\bSi=\bSi^\rig$.
\end{enumerate}
\end{proposition}

\subsection{Lifting the fan}
\label{sec:lift}

Let $\bSi=(N,\Si,\beta)$ be a stacky fan, where $N\cong \bZ^n$.
Let $U$ be defined as in Section \ref{sec:construct}.
The open embedding $U\hookrightarrow \bC^r$ is $\tT$-equivariant, and can
be viewed as a morphism between smooth toric varieties. More explicitly, consider
the $r$-dimensional cone
$$
\tsi_0 =\{ y_1 \tb_1 +\cdots + y_r \tb_r \mid y_1,\ldots, y_r  \in \bR_{\geq 0} \}\subset \tN_\bR =\tN\otimes_\bZ \bR,
$$
and let  $\tSi_0 \subset \tN_\bR$ be the fan which consists of all
the faces of $\tsi_0$. Then $\bC^r$ is the smooth toric variety defined by
the fan $\tSi_0$. We define a subfan $\tSi \subset \tSi_0$ as follows.
Given $\si\in \Si(d)$, such that $\si\cap \{ \bar{b}_1,\ldots, \bar{b}_r\}
=\{ \bar{b}_{i_1},\ldots, \bar{b}_{i_d}\}$, let
$$
\tsi = \{ y_1 \tb_{i_1}+\cdots + y_d \tb_{i_d}\mid y_1,\ldots, y_d\in \bR_{\geq 0}\} \subset \tN_\bR.
$$
Then there is a bijection $\Si\to \tSi$ given by $\si\mapsto \tsi$, and $U$ is the smooth toric
variety defined by $\tSi$.

For any $d$-dimensional cone $\tsi\in \tSi$, let $I=\{ i\mid \rho_i\subset \si\}$, and define
\begin{eqnarray*}
U_\tsi&=& \Spec\bC[\tsi^\vee\cap \tM]= \bC^r-\{\prod_{i\notin I}z_i=0\}\\
&=& \{ (z_1,\ldots,z_r)\in \bC^r\mid z_i \neq 0 \textup{ for }i\notin I\} \cong \bC^d\times (\bC^*)^{r-d},\\
\tT_{\tsi}&=& \{ (\tit_1,\ldots, \tit_r)\in \tT\mid \tit_i=1 \textup{ for } i\notin  I\} \cong (\bC^*)^d.
\end{eqnarray*}
Then $U_\tsi$ is a Zariski open subset of $U$, and $\tT_{\tsi}$ is a subtorus of $\tT$.
The action of $G_{\bSi}$ on $U_\tsi$ gives rise to a stack
$[U_\tsi/G_{\bSi}]$ denoted by $\cX_\si$, which is a substack of $\cX_{\bSi}$.
We have  $\tT$-equivariant open embeddings
$$
\tT \hookrightarrow X_{\tSi}=U \hookrightarrow X_{\tSi_0}=\bC^r.
$$
The $\tT$-equivariant line bundles on $U_\tsi=\Spec\bC[\tsi^\vee\cap \tM]$
are in one-to-one correspondence with characters in $\Hom(\tT_\tsi,\bC^*)$.
Moreover, we have canonical isomorphisms
$$
\Hom(\tT_{\tsi},\bC^*) \cong \tM/(\tsi^\perp\cap \tM) \cong M_\si.
$$
Given $\chi \in M_\si$, let $\cO_{U_\tsi}(\chi)$  denote the  $\tT$-equivariant
line bundle on $U_\tsi$ associated to $\chi\in M_\si$, and
let $\cO_{\cX_\si}(\chi)$ denote the corresponding $\cT$-equivariant line bundle on $\cX_\si= [U_{\tsi}/G_{\bSi}]$.
Let $\tchi\in\tM$ be any representative of the coset
$\chi\in \tM/(\tsi^\perp\cap \tM) \cong M_\si$. Then the $\tT$-weights of
$\Gamma(U_\tsi,\cO_{U_\tsi}(\chi)) =\Gamma(\cX_\si,\cO_{\cX_\si}(\chi))$
are in one-to-one correspondence with points
in $\tchi+(\tsi^\vee\cap \tM)$.

\section{Equivariant line bundles and Twisted Polytopes} \label{sec:twisted}

In this section, we describe equivariant line bundles on a toric DM stack $\cX$
defined by a stacky fan $\bSi=(N, \Si, \beta)$. 

\subsection{Equivariant line bundles}\label{sec:Lc}
Let $U$,  $G_{\bSi}$, $\tSi$  be defined as in Section \ref{sec:toricDM}, 
so that $U=X_\tSi$ and $\cX= [U/G_{\bSi}]$.
For $i=1,\ldots,r$, let $\tD_i\subset U$ (resp. $\tD_i'\subset \bC^r$) 
be the $\tT$-divisor defined by $z_i=0$.  
For any $\vc=(c_1,\ldots, c_r)\in \bZ^r$, let
$$
\tL_{\vc}=\cO_U\Bigl(\sum_{i=1}^r c_i \tD_i\Bigr),\quad
\tL'_{\vc}=\cO_{\bC^r}\Bigl(\sum_{i=1}^r c_i \tD_i'\Bigr).
$$
Then $\tL_{\vc}$ and $\tL_{\vc}'$ are $\tT$-equivariant line bundles
on $U$ and on $\bC^r$, respectively, and $\tL_{\vc}=\tL_{\vc}'\bigr|_U$.
The $\tT$-equivariant line bundle $\tL_{\vc}$ descends to a $\cT$-equivariant
line bundle $\cL_{\vc}$ on $\cX$. Any $\cT$-equivariant line bundle
on $\cX$ is of the form $\cL_{\vc}$ for  some $\vc=(c_1,\ldots,c_r)\in \bZ^r$.

\subsection{Twisted polytopes} \label{sec:twisted-polytope}
In this subsection, we make  the following assumptions on the fan $\Si$. 
\begin{equation}\label{eqn:top}
\begin{aligned}
& \bullet \textup{All the maximal cones in $\Si$ are $n$-dimensional, where $n=\dim_\bR N_\bR$.}\\
& \bullet \textup{We fix a total ordering $C_1,\ldots,C_v$ of the maximal cones in $\Si$.}
\end{aligned}
\end{equation}
Under the above assumptions, we will give an alternative description
of equivariant line bundles on $\cX$. 

For each $C_i$ we have 
$$
\baN_{C_i} \subset \baN \subset N_\bR,
$$
where $\baN_{C_i}\subset \baN$ is a sublattice of finite index.
Then 
$$
M:=\Hom(\baN,\bZ),\quad M_{C_i}=\Hom(\baN_{C_i},\bZ)
$$ 
can be identified with subgroups of $M_\bR$:
\begin{eqnarray*}
M&=& \{ x\in M_\bR\mid \langle x,v\rangle \in \bZ \textup{ for all } v\in \baN \},\\
M_{C_i}&=&\{ x\in M_\bR\mid \langle x, v \rangle \in \bZ \textup{ for all } v\in \baN_{C_i}\},
\end{eqnarray*}
where $\langle \ , \ \rangle: M_\bR\times N_\bR \to \bR$ is the natural pairing.
For each $C_i$ we have
$$
M\subset M_{C_i}\subset M_\bR.
$$

\begin{definition}[twisted polytope]\label{df:twisted-polytope}
Let $\bSi=(N, \Si, \beta)$ be a stacky fan satisfying \eqref{eqn:top}.
A \emph{twisted polytope} for $\bSi$ is an ordered $v$-tuple $\uchi = (\chi_1,\ldots,\chi_v)$, where
$\chi_i \in M_{C_i}$,  with the property that for any $1 \leq i \leq v$ and $1 \leq j \leq v$, the linear forms
$\langle \chi_i,-\rangle$ and $\langle \chi_j,-\rangle$ agree when restricted to $C_i \cap C_j$.
\end{definition}
The terminology is motivated by \cite{KT}.

\begin{lemma}\label{line-bundle-twisted}
Let $\cX$ be the toric DM stack defined
by a stacky fan $\bSi=(N,\Si,\beta)$ satisfying \eqref{eqn:top},
and let $p:\cX \to X$ be the map to the coarse moduli space $X=X_\Si$.
Let $T=p(\cT)\cong (\bC^*)^n$. 
\begin{enumerate}
\item[(a)] For each twisted polytope $\uchi = (\chi_1,\ldots,\chi_v)$ for $\bSi$,
there is up to isomorphism a unique $\cT$-equivariant line bundle
$\cO_{\cX}(\uchi)$ with the property that $\cO_{\cX}(\uchi)\bigr|_{\cX_{C_i}} \cong \cO_{\cX_{C_i}}(\chi_i)$,
where $\cO_{\cX_{C_i}}(\chi_i)$ is defined as in Section \ref{sec:lift}.
\item[(b)] If $\chi_i \in M\subset M_{C_i}$ for $i=1,\ldots, v$ then
there is a $T$-equivariant line bundle $\cO_X(\uchi)$ on the coarse
moduli space $X$ such that $\cO_{\cX}(\uchi)=p^*\cO_X(\uchi)$.
\end{enumerate}
\end{lemma}
\begin{proof} (a) For each $i,j \in \{1,\ldots,r\}$,
we define 
$$
C_{ij}= C_i\cap C_j,\quad U_i=U_{\tC_i},\quad U_{ij}= U_{\tC_{ij}}=U_i\cap U_j.
$$
Then $U_{ij} =U_{ji}$, and $U_{ii}=U_i$.
It suffices to show that, there is a unique $\tT$-equivariant
line bundle $\cO_U(\uchi)$ on $U$ with th property that
$\cO_U(\uchi)\bigr|_{U_i}\cong \cO_{U_i}(\chi_i)$, where $\cO_{U_i}(\chi_i)$ is defined
as in Section \ref{sec:lift}.

There is an inclusion $N_{C_{ij}}\to N_{C_i}$ which induces
a surjective map $f_{ij}: M_{C_i}\to M_{C_{ij}}$. Note that
$\langle \chi_i, -\rangle$ and $\langle \chi_j,-\rangle$ agree when
restricted to $C_{ij}$  if and only if $f_{ij}(\chi_i)= f_{ji}(\chi_j)\in M_{C_{ij}}$.
We define $\chi_{ij}= f_{ij}(\chi_i)=f_{ji}(\chi_j)\in M_{C_{ij}}$. Note
that $f_{ii}$ is the identity map and $\chi_{ii}=\chi_i$.
Then $\cO_{U_i}(\chi_i)\bigr|_{U_{ij}}$ is isomorphic to $\cO_{U_{ij} }(\chi_{ij})$
as $\tT$-equivariant line bundles on $U_{ij}$. For each $i, j$, We fix
an isomorphism of $\tT$-equivariant line bundles:
$$
\lambda_{ij}: \cO_{U_i}(\chi_i)\bigr|_{U_{ij}}\stackrel{\sim}{\longrightarrow}
\cO_{U_{ij}}(\chi_{ij}) = \cO_{U_{ji}}(\chi_{ji}).
$$ 
In particular, we take $\lambda_{ii}$ to be the identity map.
For each $i,j$, we define an isomorphism of $\tT$-equivariant line bundles:   
$$
\phi_{ij} =\lambda_{ji}^{-1}\circ \lambda_{ij}: \cO_{U_i}(\chi_i)\bigr|_{U_{ij}}
\stackrel{\sim}{\longrightarrow} \cO_{U_j}(\chi_j)\bigr|_{U_{ji}=U_{ij}}.
$$
Then
\begin{enumerate}
\item[(i)] for each $i$, $\phi_{ii}$ is the identity map, and
\item[(ii)] for each $i, j, k$, $\phi_{ik}=\phi_{jk}\circ \phi_{ij}$ on $U_i\cap U_j\cap U_k$. 
\end{enumerate}
Therefore there exists up to isomorphism a unique $\tT$-equivariant line bundle 
$\cO_U(\uchi)$ on $U$ with isomoprhisms
$\psi_i: \cO_U(\uchi)\bigr|_{U_i}\stackrel{\sim}{\longrightarrow} \cO_{U_i}(\chi_i)$ of $\tT$-equivariant
line bundles such  that $\psi_j =\phi_{ij}\circ\psi_i$ on $U_{ij}$.  

(b) The construction of the $T$-equivariant line
bundle $\cO_X(\uchi)$ on the simplicial toric variety
$X=X_\Si$ from such $\chi_1,\ldots, \chi_r \in M$ is well-known, see
for example \cite[Section 3.4]{Fu}. It 
is clear from construction that $p^*\cO_X(\uchi)= \cO_{\cX}(\uchi)$. 
\end{proof}

 Let $C_1,\ldots, C_v$ be defined as above.
Given $1\leq i_0<\ldots < i_k \leq v$, define
$C_{i_0\cdots i_k}= C_{i_0}\cap \cdots \cap C_{i_k}$.
Given a twisted polytope $\uchi=(\chi_1,\ldots,\chi_v)$ of $\bSi$,
for each $1\leq i_0< \cdots <i_k\leq v$ define $\chi_{i_0\cdots i_k}\in M_{C_{i_0\cdots i_k}}$
to be the image of $\chi_{i_0}\in M_{C_{i_0}}$ under
the group homomorphism $M_{C_{i_0}}\to M_{C_{i_0\cdots i_k}}$. Then 
$$
\{ \chi_{i_0\cdots i_k}\in M_{C_{i_0\cdots i_k}}\mid 1\leq i_0<\cdots <i_k\leq v\}
$$ 
satisfies
the following properties:
\begin{enumerate}
\item  If $1\leq j_0<\cdots <j_l\leq v$ refines $1\leq i_0<\cdots <i_k\leq v$ then
$\chi_{i_0\cdots i_k}\mapsto \chi_{j_0\cdots j_l}$ under the group homomorphism 
$M_{C_{i_0\cdots i_k}}\to M_{C_{j_0\cdots j_l}}$.
\item  $\cO_\cX(\uchi)\bigr|_{\cX_{C_{i_0\cdots i_k}}}\cong 
\cO_{\cX_{C_{i_0\cdots i_k}}}(\chi_{i_0\cdots i_k})$.
\end{enumerate}

\subsection{Equivariant $\bQ$-ample line bundles}
In this subsection we assume the toric DM stack $\cX$ is {\em complete}, i.e.,
it is defined by a stacky fan $\bSi=(N,\Si,\beta)$ where
$\Si$ is a complete fan in $N_\bR$.

Given a twisted polytope $\uchi=(\chi_1,\ldots,\chi_r)$ of $\bSi$ and $n\in \bZ$, let
$n\uchi = (n\chi_1,\ldots, n\chi_r)$. Then $n\uchi$ is a twisted polytope
of $\bSi$, and $\cO_{\cX}(n\uchi)= \cO_{\cX}(\uchi)^{\otimes n}$. 
Given any twisted polytope $\uchi=(\chi_1,\ldots,\chi_r)$ of $\bSi$ there
exists a positive integer $n$ such that
$n\chi_i\in M$ for $i=1,\ldots, v$. Then $n\uchi$ defines
an equivariant line bundle $\cO_X(n\uchi)$ on the coarse moduli space $X=X_\Si$, and
$\cO_{\cX}(\uchi)^{\otimes n}=p^*\cO_{X}(n\uchi)$.

\begin{definition}[$\bQ$-ample] \label{df:Qample}
Let $\cX$ be a complete toric DM stack, and let $p:\cX\to X$ be the morphism to the
coarse moduli space.  We say a line bundle $\cL$ on $\cX$ is {\em $\bQ$-ample}
if there exists a positive integer $n>0$ and an ample line bundle $L$ on $X$ such
that $\cL^{\otimes n} = \pi^*L$.
\end{definition}

\begin{theorem} \label{thm:fulton3}
Let $\cX=\cX_\bSi$ be a complete toric DM stack defined by a stacky fan
$\bSi=(N,\Si,\beta)$, and let $\uchi$ be a twisted polytope of $\bSi$.
The line bundle $\cO_{\cX}(\uchi)$ is $\bQ$-ample precisely when 
$\uchi$ satisfies the following two conditions:
\begin{enumerate}
\item[(i)] The set $\{\chi_1,\ldots,\chi_v\}$ is strictly convex, in the sense that 
its convex hull is strictly larger than the convex hull of any subset $\{\chi_{i_1},\ldots,\chi_{i_w}\}$.
\item[(ii)] The convex hull of $\{\chi_1,\ldots,\chi_v\}$ coincides with the set of all $\xi \in M_\bR$ satisfying
$$
\langle\xi,\gamma\rangle \geq \langle\chi_i,\gamma\rangle \text{ for all $i$ and all $\gamma \in C_i$}.
$$
\end{enumerate}
\end{theorem}
\begin{proof} Let $p:\cX\to X$ be the morphism
to the coarse moduli space $X=X_\Si$.
We have seen that there exists $n\in \bZ_{>0}$ such that $n\uchi$ defines
an equivariant line bundle $\cO_X(n\uchi)$ on the coarse moduli space $X$, and that
$\cO_{\cX}(\uchi)^{\otimes n} =p^*\cO_X(n\uchi)$. By \cite[Section 3.4]{Fu},
$\cO_X(n\uchi)$ is ample if and only if $n\uchi$ satisfies the
above conditions (i) and (ii). The proof is completed by
the following observation: $\uchi$ satisfies conditions (i) and (ii) 
$\Leftrightarrow$ $n\uchi$ satisfies conditions (i) and (ii) for all $n\in \bZ_{>0}$ 
$\Leftrightarrow$ $n\uchi$ satisfies conditions (i) and (ii) for
some $n\in \bZ_{>0}$. 
\end{proof}

\section{Standard Quasicoherent Sheaves and Costandard Constructible Sheaves}
\label{sec:Theta-Theta}

Let $\cX_{\Si}$ be a toric DM stack defined by a stacky fan $\bSi=(N,\Si,\beta)$.
In this section, we introduce a useful class $\{\Theta(\si,\chi)\}$ of constructible sheaves on $M_\bR$,
and a corresponding class $\{\Theta'(\si,\chi)\}$ of $\cT$-equivariant
quasicoherent sheaves on $\cX_\bSi=[U/G_{\bSi}]$, and show that the dg
categories they generate are quasi-equivalent to each other.
These two classes of sheaves are closely related to the class of constructible sheaves
$\{\Theta(\tsi,\chi)\}$ on $\tM_\bR$ and the class of $\tT$-equivariant quasicoherent sheaves
$\{\Theta'(\tsi,\chi)\}$ on the smooth toric variety $U$ introduced in \cite[Section 3]{more}.

\subsection{The poset $\Ga(\bSi)$.}
\label{sec:poset}

\begin{definition}\label{df:preorder}
Let $\bSi=(N,\Si,\beta)$ be a stacky fan. Define
$$
\Ga(\bSi)=\{ (\si,\chi)\mid \si \in \Si, \chi\in M_\si \}.
$$
\begin{enumerate}
\item Given a pair $(\si,\chi)\in \Ga(\bSi)$, where $\si\in \Si(d)$, we have
$$
\si\cap \{\bab_1,\ldots,\bab_r\}=\{ w_1,\ldots,w_d\},
$$
where $\bab_1,\ldots,\bab_r\in \baN \subset N_\bR$ are define
as in Section \ref{sec:stacky}.
\begin{eqnarray*}
\si^\vee_\chi &=&\{ x\in M_\bR \mid \langle x, w_i\rangle \geq \langle \chi, w_i\rangle_\si\}\\
(\si^\vee_\chi)^\circ &=& \{ x\in M_\bR\mid \langle x, w_i\rangle > \langle \chi, w_i\rangle_\si\}\\
\si^\perp_\chi &=& \{ x\in M_\bR \mid \langle x, w_i\rangle =\langle \chi,w_i\rangle_\si \}
\end{eqnarray*}
where  $$
\langle\  ,\ \rangle: M_\bR \times N_\bR \to \bR,\quad
\langle\ , \ \rangle_\si: M_\si\times \baN_\si \to \bZ
$$
are the natural pairings.
\item Give the set  $\Gamma(\bSi)$ of ordered pairs $(\si,\chi)$ a partial order, by setting
$$
(\si_1,\chi_1) \leq (\si_2,\chi_2)
$$
whenever $(\si_1)_{\chi_1}^\vee \subset (\si_2)_{\chi_2}^\vee$.
\item Let $\Gamma(\bSi)_\bC$ be the $\bC$-linear category whose objects are the elements of $\Gamma(\bSi)$, with $\dghom\big( (\si_1,\chi_1),(\si_2,\chi_2)\big)$ a one- or zero-dimensional vector space depending on whether $(\si_1,\chi_1) \leq (\si_2,\chi_2)$, and with the evident composition rule.  We will regard it as a dg category with the Hom complexes concentrated in degree zero.
\end{enumerate}
\end{definition}

It is clear from the definitions that
$$
\Ga(\bSi)=\Ga(\bSi^\rig),\quad \Ga(\bSi)_{\bC}=\Ga(\bSi^\rig)_{\bC}.
$$

The smooth toric variety $U$ is defined by the fan $\tSi\subset\tN_\bR$. In \cite[Section 3.1]{more}, we
define a poset 
$$
\Ga(\tSi, \tM)=\{ (\tsi,\chi)\mid \si\in \tSi, \chi\in \tM/(\si^\perp\cap \tM) \}.
$$
Recall that $\tM/(\si^\perp\cap \tM) \cong M_\si$.
For any $(\tsi_1,\chi_1), (\tsi_2,\chi_2) \in \Ga(\tSi,\tM)$, 
$(\tsi_1,\chi_1)\leq (\tsi_2,\chi_2)$ in $\Ga(\tSi,\tM)$
if and only if $(\si_1,\chi_1)\leq (\si_2,\chi_2)$ in 
$\Ga(\bSi)$. We conclude:

\begin{lemma} \label{poset-updown}
The bijective map $\Ga(\tSi,\tM) \to \Gamma(\bSi)$,
$(\tsi,\chi)\mapsto (\si,\chi)$ is an isomorphism
of posets, and induces a quasi-equivalence of
$dg$ categories
$$
\Ga(\tM,\tSi)_\bC \cong \Ga(\bSi)_\bC.
$$
\end{lemma}

\subsection{Costandard sheaves on cones}
\label{sec:Theta}

In this section we introduce constructible sheaves $\Theta(\si,\chi)$ on $M_\bR$, 
indexed by elements of $\Ga(\bSi)$.  They are \emph{costandard} in the sense of \cite{NZ}.

\begin{definition}
Given  $\si \in \Si$ and $\chi \in M_\si$, let $j=j_{(\si^\vee_\chi)^\circ}:
(\si_\chi^\vee)^\circ\hookrightarrow M_\bR$ be the inclusion map, and define
$$
\Theta(\si,\chi) := j_!\omega_{(\si_\chi^\vee)^\circ} \in \Ob(\Sh_c(M_\bR)).
$$
\end{definition}

We recall the definition of similar constructible sheaves
$\Theta(\tsi,\chi)$ on $\tM_\bR$ for the fan $\tSi$ \cite[Definition 3.1]{more}.
Given $\tsi\in \tSi$ and $\chi\in \tM/(\tsi^\perp\cap\tM)\cong M_\si$,
define
$$
\Theta(\tsi,\chi) = \tj_!\omega_{(\chi+\tsi^\vee)^\circ}\in Ob(Sh_c(\tM_\bR)),
$$
where
$(\chi+\tsi^\vee)^\circ$ is the interior of the shifted dual cone
$\chi+\tsi^\vee \subset \tM_\bR$, and
$\tj :(\chi+\tsi^\vee)^\circ\hookrightarrow M_\bR$ is the inclusion.

The group homomorphism $\beta:\tN\to N$ induces
$$
\beta^*:M=\Hom(N,\bZ)=\Hom(\bar{N},\bZ) \longrightarrow \tM=\Hom(\tN,\bZ).
$$ 
Let $\hat{\beta}=\beta^*\otimes\bR :M_\bR\to \tM_\bR$.
Then $\hat{\beta}$ is an injective $\bR$-linear map.
\begin{lemma}\label{Theta-updown}
\begin{equation}\label{eqn:Theta-updown}
\hat{\beta}^! \Theta(\tsi,\chi) = \Theta(\si,\chi).
\end{equation}
\end{lemma}
\begin{proof} Let
$j:(\si_\chi^\vee)^\circ\to M_\bR$ and $\tj:(\chi+\tsi^\vee)^\circ\to \tM_\bR$ be defined as above,
and let $\hat{\beta}': (\si_\chi^\vee)^\circ \to (\chi+\tsi^\vee)^\circ$ be the restriction
of $\hat{\beta}$. Then $j$, $\tj$ are open embeddings,
$\hat{\beta}$, $\hat{\beta}'$ are closed embeddings, and
$\tj \circ \hat{\beta}' = \hat{\beta}\circ j$.
$$
\hat{\beta}^! \tj_!\omega_{(\chi+\tsi^\vee)^\circ} 
= j_! (\hat{\beta}')^!\omega_{(\chi+\tsi^\vee)^\circ}
= j_!\omega_{(\si_\chi^\vee)^\circ}. 
$$
\end{proof}

\begin{proposition} \label{prop:cone}
For any  $(\si,\phi), (\tau,\psi)\in \Gamma(\bSi)$
$$
\Ext^i(\Theta(\si,\phi),\Theta(\tau,\psi)) =
\begin{cases}
\bC & \text{if $i = 0$ and } \si_\phi^\vee \subset \tau_\psi^\vee \\
0 &\text{otherwise}
\end{cases}
$$
\end{proposition}

\begin{proof} The proof is the same as the proof of \cite[Proposition 3.3 (1)]{more}.
\end{proof}

\subsection{Standard equivariant quasicoherent sheaves on $\cX_{\bSi}$}
\label{sec:Theta-prime}

In this section we introduce $\cT$-equivariant quasicoherent sheaves $\Theta'(\si,\chi)$
on the toric DM stack $\cX_{\bSi}=[U/G_\bSi]$, indexed by $\Ga(\bSi)$.
Under the quasi-equivalence $\cQ_\cT(\cX_{\bSi}) \cong \cQ_{\tT}(U)$
they correspond to $\tT$-equivariant quasicoherent sheaves
$\Theta'(\tsi,\chi)$ on $U$, indexed by $\Ga(\tSi,\tM)$.

We first recall the definition of $\Theta'(\tsi,\chi)$
in \cite[Section 3.2]{more}, using the notation in this paper.
Given $\tsi\in \tSi$ and $\chi\in \tM/(\si^\perp\cap \tM)\cong M_\si$, define
$\cO_{U_\tsi}(\chi)$ as in Section \ref{sec:lift}.  Let $\iota_{\tsi}:U_\tsi\to U$
be the open embedding. We define
$$
\Theta'(\tsi,\chi)= \iota_{\tsi*}\cO_{U_\si}(\chi)\in Ob(\cQ_{\tT}(U)).
$$

We now define the sheaves $\Theta'(\si,\chi)$ on $\cX_{\bSi}$.
\begin{definition}
Given $(\si,\chi)\in \bSi$, let $j_\si: \cX_\si\to \cX_\bSi$ be the open embedding, and define
$$
\Theta'(\si,\chi)= j_{\si*}\cO_{\cX_\si}(\chi)\in Ob(\cQ_\cT(\cX_\bSi)).
$$
\end{definition}

\begin{lemma}\label{Theta-prime-updown}
Let $q$ denote the dg functor $\cQ_{\tT}(U) \to \cQ_\cT(\cX_{\bSi})$ (which
is a quasi-equivalence of dg categories). Then
\begin{equation}\label{eqn:Theta-prime-updown}
q(\Theta'(\tsi,\chi))= \Theta'(\si, \chi).
\end{equation}
\end{lemma}
\begin{proof}
We have a 2-cartesian diagram
$$
\begin{CD}
U_\tsi  @>{\iota_\tsi}>>  & U\\
@V{\pi_\si}VV &  @VV{\pi}V \\
\cX_\si @>{j_\si}>> &\cX_{\bSi}.
\end{CD}
$$
where $\iota_\tsi$ and $j_\si$ are open embeddings. 
We need to show $\pi^* \Theta'(\si,\chi)
=\Theta'(\tsi,\chi)$.
$$
\pi^* \Theta'(\si,\chi)
=\pi^* j_{\si *}\cO_{\cX_{\si}}(\chi)
= \iota_{\tsi*}\pi_\si^* \cO_{\cX_{\si}}(\chi)
=\iota_{\tsi*}\cO_{U_{\tsi}}(\chi)
=\Theta'(\tsi,\chi).
$$
\end{proof}

\begin{proposition} \label{prop:stand}
For any  $(\sigma,\phi), (\tau,\psi)\in \Ga(\bSi)$ we have
$$
\Ext^i(\Theta'(\sigma,\phi),\Theta'(\tau,\psi)) = \begin{cases}
\bC & \text{if $i = 0$ and $\sigma_\phi^\vee \subset \tau_\psi^\vee$} \\
0 &\text{otherwise}
\end{cases}
$$
where the Ext group is taken in the category $\cQ_\cT(\cX_\bSi)$.
\end{proposition}

\begin{proof}
Applying \cite[Proposition 3.3(2)]{more} to the smooth toric
variety $U$, we obtain the following statement:
\begin{itemize}
\item For any  $(\tsi,\phi), (\tilde{\tau},\psi)\in \Ga(\tSi,\tM)$,
$$
\Ext^i(\Theta'(\tsi,\phi),\Theta'(\tilde{\tau},\psi))
= \begin{cases}
\bC & \text{if $i = 0$ and $\phi+\tsi^\vee \subset \psi+ \tilde{\tau}^\vee$}, \\
0 &\text{otherwise}.
\end{cases}
$$
\end{itemize}
The proposition follows from Lemma \ref{Theta-prime-updown} and the above statement.
\end{proof}

\subsection{Equivalence between categories of $\Theta$-sheaves}
\label{sec:first}

\begin{definition}
Let $\ltr_{\bSi}$, $\ltrp_{\bSi}$, $\ltr_{\tSi}$, $\ltrp_{\tSi}$
be the full triangulated dg subcategories of
$\Sh_c(M_\bR)$, $\cQ_T(\cX_\bSi)$, $\Sh_c(\tM_\bR)$, $\cQ_{\tT}(U)$
generated by
$$
\begin{array}{ll}
\{ \Theta(\si,\chi) \}_{ (\si,\chi)\in \Ga(\bSi) }, 
& \{ \Theta'(\si,\chi)\}_{ (\si,\chi)\in \Ga(\bSi) }, \\ 
\{ \Theta(\tsi,\chi)\}_{ (\si,\chi)\in \Ga(\tSi,\tM) },  
& \{ \Theta'(\tsi,\chi)\}_{ (\si,\chi)\in \Ga(\tSi,\tM)},  
\end{array}
$$
respectively.
\end{definition}

\begin{theorem}\label{thm:mainfirst}
The following square of functors commutes up to natural isomorphism:
\begin{equation}\label{eqn:kappa-updown}
\begin{CD}
\langle\Theta'\rangle_{\tSi} @>{\kappa_{\tSi}}>>  & \langle \Theta\rangle_{\tSi}\\
@V{q}VV &  @V{\hat{\beta}^!}VV \\
\langle\Theta'\rangle_{\bSi} @>{\kappa_{\bSi}}>> & \langle \Theta\rangle_{\bSi}
\end{CD}
\end{equation}
where $\hat{\beta}^!$ is given by \eqref{eqn:Theta-updown},
$q$ is given by \eqref{eqn:Theta-prime-updown}, and
$\kappa_\tSi$ and $\kappa_{\bSi}$ are given by
$$
\kappa_\tSi(\Theta'(\tsi,\chi)) = \Theta(\si,\chi),\quad
\kappa_\bSi(\Theta'(\si,\chi))=\Theta(\si,\chi).
$$
Moreover, the dg functors $\hat{\beta}^!$, $q$, $\kappa_{\tSi}$,
$\kappa_{\bSi}$ in the above diagram \eqref{eqn:kappa-updown}
are quasi-equivalences of triangulated dg categories.
\end{theorem}

\begin{proof} This follows from  Proposition \ref{prop:cone}, Proposition \ref{prop:stand},
and \cite[Proposition 3.3]{more}. The
four categories in the above diagram \eqref{eqn:kappa-updown} are quasi-equivalent to
$$
\Tr(\Ga(\bSi)_\bC)\cong \Tr( \Ga(\tSi,\tM)_\bC ).
$$
\end{proof}

\subsection{Coherent-constructible dictionary---line bundles}
\label{sec:ccc-line}

Let $\Perf_\cT(\cX_\bSi)$ denote the dg category of perfect
complexes of $\cT$-equivariant coherent sheaves on $X_\bSi$.
The dg functor $\cQ_{\tT}(U)\to \cQ_\cT(\cX_{\bSi})$ restricts
to the dg functors
$$
\Perf_{\tT}(U)\to \Perf_\cT(\cX_{\bSi}),\quad
\langle\Theta'\rangle_{\tSi}\to \langle \Theta'\rangle_{\bSi}
$$
which are quasi-equivalences of triangulated dg categories. The proof of
\cite[Corollary 3.5]{more} shows that
$\Perf_{\tT}(U)\subset \langle \Theta'\rangle_{\tSi}$. Therefore
$\Perf_\cT(\cX_\bSi)\subset \langle\Theta'\rangle_{\bSi}$, and we have:
\begin{corollary} \label{cor:perfect}
The functor $\kappa_{\bSi}$ defines a full embedding of
$\Perf_\cT(\cX_\bSi)$ into $\Sh_c(M_\bR)$.
\end{corollary}

Thus to each vector bundle we can associate a complex of sheaves on $M_\bR$.
For the rest of this section we assume that $\Si$ is {\em complete}, and we investigate
this association in more detail for line bundles. Given a twisted polytope
$\uchi=(\chi_1,\ldots,\chi_v)$ of $\bSi$, defined as in Section \ref{sec:twisted-polytope},
define 
$$
\{\chi_{i_0\cdots i_k}\in M_{C_{i_0\cdots i_k}}\mid 1\leq i_0<\cdots <i_k\leq v\}
$$
as in the last paragraph of Section \ref{sec:twisted-polytope}.
Then whenever $1\leq j_0 < j_1 < \ldots < j_\ell\leq v$ refines 
$1\leq i_0 < \ldots< i_k\leq v$, we have a well-defined inclusion map
$\Theta(C_{i_0 \ldots i_k},\chi_{i_0 \ldots i_k}) \hookrightarrow \Theta(C_{j_0 \cdots j_\ell},\chi_{j_0 \ldots j_\ell})$.

\begin{definition}\label{Puchi}
For each twisted polytope $\uchi$, let $P(\uchi)\in Sh_c(M_\bR)$ be the following 
cochain complex
\begin{equation}\label{eqn:Puchi}
\bigoplus_{i_0} \Theta(C_{i_0},\chi_{i_0})
\to
\bigoplus_{i_0 < i_1} \Theta(C_{i_0 i_1},\chi_{i_0 i_1})
\to
\cdots
\end{equation}
where the differential is the alternating sum of inclusion maps.  
\end{definition}

Naively, the first term $\bigoplus_{i_0} \Theta(C_{i_0},\chi_{i_0})$
of \eqref{eqn:Puchi}  would be in degree zero, but because 
$\Theta(\si,\chi)=j_!\omega_{(\si_\chi^\vee)^\circ}
=j_!\fofr_{(\si_\chi^\vee)^\circ}[\dim M_\bR]$, 
$P(\uchi)$ is isomorphic to a complex of sheaves whose first term  is
in degree $-\dim M_\bR$.

\begin{theorem}\label{thm:algebraic}
Let $\cX=\cX_{\bSi}$ be a complete toric DM stack defined
by a stacky fan $\bSi$. Let $\cO_{\cX}(\uchi)$ denote the $\cT$-equivariant
line bundle on $\cX$ associated to  a twisted polytope $\uchi$ of $\bSi$, 
and let $P(\uchi)\in \Sh_c(M_\bR)$  be as in Definition \ref{Puchi}. Then:
\begin{enumerate}
\item $\kappa_{\bSi}(\cO_{\cX}(\uchi))\cong P(\uchi)$.
\item  Denote the convex hull of $\{\chi_1,\ldots,\chi_v\}$ by $\bfP$ and its interior by $\bfP^\circ$.
If $\cO_{\cX}(\uchi)$ is ample, then $P(\uchi)\cong j_!\omega_{\bfP^\circ}$, where
$j: \bfP^\circ \hookrightarrow M_\bR$ is the inclusion map.
Therefore the embedding functor $\kappa_{\bSi}:\Perf_{\cT}(\cX_{\bSi})\hookrightarrow
\Sh_c(M_\bR)$ maps the sheaf $\cO_{\cX_{\bSi}}(\uchi)$ to  $j_!\omega_{\bfP^\circ}$,
the costandard constructible  sheaf on $\bfP^\circ$.
\end{enumerate}
\end{theorem}
\begin{proof} To see (1), apply $\kappa_\bSi$ to the \v{C}ech resolution of  $\cO_{\cX}(\uchi)$
(cf. the proof of \cite[Corollary 3.5]{more}).
The proof of (2) is a minor modification of the proof of
\cite[Theorem 3.7]{more}.
\end{proof}

\subsection{Morphisms between toric DM stacks}
\label{sec:mor}

Following \cite[Remark 4.5]{BCS}, we introduce the following definition.

\begin{definition}
Let $\bSi_1 =(N_1,\Si_1,\beta_1)$ and $\bSi_2 =(N_2, \Si_2,\beta_2)$ 
be stacky fans. A {\em morphism}
$f:\bSi_1\to \bSi_2$ is a group homomorphism
$f:N_1\to N_2$ such that
\begin{itemize}
\item  For any cone $\si_1\in \Si_1$ there exists a cone $\si_2\in \Si_2$
such that $f_\bR(\si_1)\subset \si_2$, where $f_\bR = f\otimes_\bZ \bR: N_{1,\bR} \to N_{2,\bR}$.
\item If $\si_1\in \Si_1$, $\si_2\in \Si_2$, and $f_\bR(\si_1)\subset \si_2$,
then $f(N_{\si_1}) \subset N_{\si_2}$, where $N_{\si_i}$ is the subgroup
of $N_i$ defined as in Section \ref{sec:stacky}.
\end{itemize}
\end{definition}

A morphism $f:\bSi_1\to \bSi_2$ induces (see \cite[Remark 4.5]{BCS})
\begin{itemize}
\item a map $\cT_1 \to \cT_2$ 

\item a map $u_{f,\sigma_1,\sigma_2}:\cX_{\sigma_1} \to \cX_{\sigma_2}$ for
a pair of cones $\sigma_1$, $\sigma_2$ such that
$f_\bR(\sigma_1)\subset \sigma_2$.
\item a map $u=u_f:\cX_1=\cX_{\bSi_1}\to \cX_2= \cX_{\bSi_2}$ assembled from
$u_{f,\sigma_1,\sigma_2}$, which extends the map $\cT_1\to \cT_2$,
and is equivariant; we have the following 2-cartesian diagram:
$$
\xymatrix{
\cX_1\times \cT_1 \ar[r]^{(u, u|_{\cT_1})}  \ar[d]_{a_1}  &
\cX_2\times \cT_2 \ar[d]_{a_2} \\
\cX_1  \ar[r]^u & \cX_2
}
$$
where $a_i$ is the $\cT_i$-action on $\cX_i$.
\item a linear map $v=v_f: M_{2,\bR}\to M_{1,\bR}$ of real vector spaces.
\end{itemize}

\begin{remark}
In \cite{Iw1}, I. Iwanari established an equivalence between the 2-category
of toric stacks and the 1-category of stacky fans. In  
\cite[Definition 2.1]{Iw1}, $N$ is a free abelian group, so toric stacks
in \cite{Iw1} are toric orbifolds. 
\end{remark}

When the source $\cX_1$ is a {\em complete} toric DM stack, i.e.,
the coarse moduli space $X_1$ of $\cX_1$ is a complete
simplicial toric variety, F. Perroni gave a description of morphisms $\cX_1\to \cX_2$ in terms of
homogeneous polynomials \cite[Theorem 5.1]{Pe}. This description is similar to
Cox's description of morphisms from a complete toric variety to
a smooth toric variety \cite[Theorem 3.2]{Co}.

Suppose that $\cX_i =[U_i/G_{\bSi_i}]$, where $U_i = \bC^{r_i}-Z(I_{\Si_i})$.
By \cite[Theorem 5.1]{Pe}, there exists a map
$$
F: \bC^{r_1}\to \bC^{r_2},\quad
z=(z_1,\ldots,z_{r_1})\mapsto (P_1(z),\ldots, P_{r_2}(z))
$$
where $P_1,\ldots, P_{r_2}\in \cO(\bC^{r_1})=\bC[z_1,\ldots,z_{r_1}]$ are
homogeneous polynomials, such that
\begin{itemize}
\item $F(U_1) \subset U_2$
\item the following diagram is 2-commutative
$$
\xymatrix{
U_1 \ar[r]^{\tu}  \ar[d]  & U_2 \ar[d] \\
\cX_1  \ar[r]^u & \cX_2
}
$$
where $\tu$ is the restriction of $F$, and
the vertical arrows are the quotient maps.
\end{itemize}
Moreover, $\{ P'_i \}$ and $\{ P_i\}$ determine 2-isomorphic morphisms
if and only if there exists $g\in G_{\bSi_2}$ such that
$$
(P_1',\ldots, P_{r_2}') = g\cdot (P_1,\ldots, P_{r_2}).
$$
Note that for a given choice of $\{P_i\}$,
$\tu:U_1\to U_2$ can be viewed as a morphism between
smooth toric varieties. We have a cartesian diagram:
$$
\xymatrix{
U_1\times \tT_1 \ar[r]^{(\tu, \tu|_{\tT_1}) }  \ar[d]_{\ta_1}  &
U_2\times \tT_2 \ar[d]_{\ta_2} \\
U_1  \ar[r]^{\tu} & U_2
}
$$
where $\tT_i\cong (\bC^*)^{r_i}$. This gives a group homomorphism
$\tf: \tN_1\to \tN_2$ which fits in the following commutative diagram.
$$
\xymatrix{
\tN_1 \ar[r]^{\tf}  \ar[d]_{\beta_1}  &
\tN_2 \ar[d]_{\beta_2} \\
N_1  \ar[r]^{f} & N_2
}
$$

\subsection{Coherent-constructible dictionary---functoriality and tensoriality}
\label{sec:ccfunctoriality}

In this section we show that the equivalence $\kappa_{\bSi}$ between
coherent and constructible sheaves intertwines with appropriate
pull-back and and push-forward functors. We use the notation in
Section \ref{sec:mor}.

\begin{theorem}[functoriality]
\label{thm:functoriality}
Let $f:\bSi_1=(N_1,\Si_1, \beta_1)\to \bSi_2=(N_1,\Si_2,\beta_2)$
be a morphism of stacky fans, where $\Si_1$ is a complete fan.
Suppose that $f$ furthermore satisfies the following conditions:
\begin{enumerate}
\item[(i)] The inverse image of any cone $\si_2\subset \Si_2$ is a union of cones in $\Si_1$.
(For instance, if both fans are complete then $f$ automatically satisfies
this condition.)
\item[(ii)] $f$ is injective.
\end{enumerate}
Let $u$, $v$, $\cX_i$, $\cT_i$ be as in Section \ref{sec:mor}. Then
\begin{enumerate}
\item The pullback $u^*:\cQ_{\cT_2}(\cX_2) \to \cQ_{\cT_1}(\cX_1)$ takes 
$\ltrp_{\bSi_2}$ to $\ltrp_{\bSi_1}$.

\item The proper pushforward $v_!:Sh_c(M_{2,\bR}) \to Sh_c(M_{1,\bR})$ takes $\ltr_{\bSi_2}$ to
$\ltr_{\bSi_1}$.

\item The following square of functors commutes up to natural isomorphism:
$$\xymatrix{
\ltrp_{\bSi_2} \ar[r]^{\kappa_2} \ar[d]_{u^*} &\ltr_{\bSi_2} \ar[d]^{v_!} \\
\ltrp_{\bSi_1} \ar[r]^{\kappa_1} &  \ltr_{\bSi_1}
}$$
where $\kappa_i=\kappa_{\bSi_i}$, $i=1,2$.
\end{enumerate}
\end{theorem}

\begin{proof}
As in Section \ref{sec:mor}, we choose liftings
$\tSi_i\subset \tN_{i,\bR}$ and $\tf:\tSi_1\to \tSi_2$ which
induces a morphism $\tu: U_1\to U_2$ of smooth toric varieties such that
$$
\xymatrix{
U_1 \ar[r]^{\tu} \ar[d] & U_2 \ar[d]\\
\cX_1 \ar[r]^u & \cX_2
}
$$
where $U_i=X_{\tSi_i}$ and $\cX_i=[U_i/G_{\bSi_i}]$.

For each cone $\sigma_2 \in \Sigma_2$, we have a 2-cartesian square
\begin{equation}\label{eqn:j}
\xymatrix{
u^{-1}(\cX_{\sigma_2}) \ar[r]^{u'} \ar[d]_j & \cX_{\sigma_2} \ar[d]^{j_{\si_2}} \\
\cX_1 \ar[r]^u & \cX_2
}
\end{equation}
where $u'= u|_{u^{-1}(\cX_{\sigma_2})}$.

We have the following cartesian square which corresponds to
\eqref{eqn:j}:
\begin{equation}\label{eqn:tj}
\xymatrix{
\tilde{u}^{-1}(U_{\sigma_2}) \ar[r]^{\tu'} \ar[d]_{\tj} & U_{\sigma_2} \ar[d]^{\tj_{\si_2}} \\
U_1 \ar[r]^{\tu} & U_2
}
\end{equation}
where $\tu'=\tu|_{\tu^{-1}(U_{\si_2})}$. Let $\cO_{U_{\sigma_2}}(\chi_2)$ be defined as
in Section \ref{sec:lift}. The vertical arrows in \eqref{eqn:tj} are open inclusions,
so by the flat base change formula we have
$$
\tu^* \tj_{\si_2 *}\cO_{U_{\si_2}}(\chi_2)  \cong \tj_* \cL
$$
where $\cL :=(\tu')^* \cO_{U_{\si_2}}(\chi_2)$ is a $\tT_1$-equivariant line bundle on $\tu^{-1}(U_{\si_2})$.

We fix a total order on the set of maximal
cones $B_1,\ldots, B_w$ contained in $f^{-1}(\si_2)$. By assumption
$$
u^{-1}(\cX_{\sigma_2}) = \bigcup_{i=1}^w \cX_{B_i},\quad
\tu^{-1}(U_{\sigma_2}) =\bigcup_{i=1}^w U_{B_i}.
$$
For each $B_i$ we have $f(N_{B_i})\subset N_{\sigma_2}$. Let
$f_{B_i,\sigma_2}: N_{B_i}\to N_{\sigma_2}$ be the restriction
of $f$, and let $f_{B_i,\sigma_2}^*: M_{\sigma_2}\to M_{B_i}$
be the dual map of $\bar{f}_{B_i,\si_2}:\baN_{B_i}\to \baN_{\si_2}$. 
Let $\phi_i = f_{B_i,\sigma_2}^*(\chi_2) \in M_{B_i}$. Then
$$
\cL|_{U_{B_i}} = \cO_{U_{B_i}}(\phi_i).
$$
More generally, put
$$
B_{i_0 \cdots i_k}= B_{i_0}\cap \cdots \cap B_{i_k},\quad
\phi_{i_0 \cdots i_k} = f_{B_{i_0 \cdots i_k},\si_2}^*(\chi_2) \in M_{B_{i_0\cdots i_k}}
$$
then $\cL|_{U_{B_{i_0\cdots i_k}} } = \cO_{U_{B_i}}(\phi_{i_0 \cdots i_k})$.
Therefore $\tj_*\cL$ is quasi-isomorphic to the complex
$$
\bigoplus_{i_0} \tj_{B_{i_0} *}\cO_{U_{B_{i_0}}}(\phi_{i_0}) \to
\bigoplus_{i_0 < i_1} \tj_{B_{i_0 i_1} *}\cO_{U_{B_{i_0 i_1}}}(\phi_{i_0 i_1}) \to \cdots
$$
Equivalently,  $u^*\Theta'(\sigma_2,\chi_2)$ is quasi-isomorphic to the complex
$$
\bigoplus_{i_0} \Theta'(B_{i_0},\phi_{i_0}) \to
\bigoplus_{i_0 < i_1} \Theta'(B_{i_0 i_1}, \phi_{i_0 i_1}) \to \cdots
$$
This proves the assertion (1).

After Theorem \ref{thm:mainfirst}, the assertions (2) and (3) follow from the commutativity of the following diagram:
$$
\xymatrix{
\ltrp_{\bSi_2} \ar[r]^{\kappa_2} \ar[d]_{u^*} &\Sh_c(M_{2,\bR}) \ar[d]^{v_!} \\
\ltrp_{\bSi_1} \ar[r]_{\kappa_1} &  \Sh_c(M_{1,\bR})
}$$

We follow the strategy of the proof of \cite[Theorem 3.8]{more}.
To construct a natural quasi-isomorphism $\iota:v_! \circ \kappa_2 \stackrel{\sim}{\to} \kappa_1 \circ u^*$, it suffices to give maps
$$
\iota_{\si_2,\chi_2}:v_! \kappa_2(\Theta'(\sigma_2,\chi_2)) \to \kappa_1(u^* \Theta'(\sigma_2,\chi_2))
$$
with the following properties.
\begin{itemize}
\item Each $\iota_{\si_2,\chi_2}$ is a quasi-isomorphism.
\item The following square commutes whenever $(\sigma_2,\chi_2) \leq (\tau_2,\psi_2)$:
$$
\xymatrix{
v_! \kappa_2(\Theta'(\si_2,\chi_2)) \ar[r] \ar[d]_{\iota_{\si_2,\chi_2}} & 
v_! \kappa_2(\Theta'(\tau_2,\psi_2)) \ar[d]_{\iota_{\tau_2,\psi_2}}\\
\kappa_1(u^* \Theta'(\si_2,\chi_2)) \ar[r] & \kappa_1(u^*\Theta'(\tau_2,\psi_2))
}
$$
\end{itemize}

As in the proof of \cite[Theorem 3.8]{more}, we have a quasi-isomorphism 
$$
v_! \kappa_2(\Theta'(\sigma_2,\chi_2))
\stackrel{\sim}{\to} j_{v( (\sigma_2)^\vee_{\chi_2})^\circ!}\omega,\quad  \omega =\omega _{((\si_2^\vee)_{\chi_2})^\circ}.
$$
Now let us compute $\kappa_1 u^* \Theta'(\sigma_2,\chi_2)$.
We have already seen that $u^*\Theta'(\sigma_2,\chi_2)$ has the \v Cech resolution
$$
\bigoplus_{i_0} \Theta'(B_{i_0},\phi_{i_0}) \to
\bigoplus_{i_0 < i_1} \Theta'(B_{i_0 i_1},\phi_{i_0 i_1}) \to \cdots
$$
After applying $\kappa_1$ we have
$$
\bigoplus_{i_0} \Theta(B_{i_0},\phi_{i_0} ) \to
\bigoplus_{i_0 < i_1} \Theta(B_{i_0 i_1},\phi_{i_0 i_1}) \to \cdots
$$
Now we define the map $\iota:v_! \kappa_2(\Theta'(\sigma_2,\chi_2)) \to \kappa_1 u^* \Theta'(\sigma_2,\chi_2)$
to be the morphism of complexes:
$$
\begin{CD}
j_{v( (\si_2)_{\chi_2}^\vee)^\circ !}\omega   @>>>   0 @>>>  \cdots \\
@VVV @VVV\\
\bigoplus_{i_0} \Theta(B_{i_0},\phi_{i_0}) @>>> \bigoplus_{i_0 < i_1} \Theta(B_{i_0 i_1},\phi_{i_0 i_1}) @>>> \cdots
\end{CD}
$$
where the nonzero vertical arrow is the direct sum of the maps induced by the inclusion of open sets
$$ v( (\sigma_2)^\vee_{\chi_2})^\circ \subset (B_{i_0})_{\phi_{i_0}}^\circ$$
This map clearly has the desired naturality property. By the argument in the last part
of the proof of \cite[Theorem 3.8]{more}, it is a quasi-isomorphism.
\end{proof}

\begin{example}
\label{ex:diagonal}
The diagonal map $N \to N \oplus N$ satisfies the hypotheses of Theorem \ref{thm:functoriality}.
The corresponding map $u:\cX \to \cX \times \cX$ is also the diagonal map, and the corresponding map
$v:M_\bR \times M_\bR \to M_\bR$ is the addition map.

More generally, let $\bSi = (N,\bSi,\beta)$ be the stacky fan defining $\cX$, where
$\Si$ is not necessarily complete. We have
$$
\xymatrix{
U \ar[r]^{\tilde{\Delta}} \ar[d] & U\times U \ar[d] \\
\cX \ar[r]^\Delta & \cX \times \cX
}
$$
where $\cX = [U/G_{\bSi}]$ and $\cX \times \cX = [U/G_{\bSi}]\times [U/G_{\bSi}] \cong [U\times U/G_{\bSi}\times G_{\bSi}]$.
Since we can lift the diagonal map $\Delta:\cX \to \cX\times \cX$ to
the diagonal map $\tilde{\Delta}: U\to U\times U$ for
any toric DM stack $\cX$, the proof of Theorem \ref{thm:functoriality} is valid, and
the conclusions of Theorem \ref{thm:functoriality} hold, when $f$ is the diagonal
map for any toric DM stack.
\end{example}

Recall that the \emph{convolution} of two sheaves $F$ and $G$ on a vector space $M_\bR$ is given by the formula $F \star G = v_!(F \boxtimes G)$, where $v$ denotes the addition map as in the example.  Convolution defines a monoidal  structure on $\Sh(M_\bR)$ 
and various subcategories, including $\ltr_{\bSi}$ and $\Sh_{cc}(M_\bR)$.  From Example \ref{ex:diagonal}, following the argument in
the proof of \cite[Corollary 3.13]{more}, we see the following:

\begin{corollary}
\label{cor:tensor1}
For any stacky fan $\bSi=(N,\Si,\beta)$,
the equivalence $\kappa_{\bSi}:\ltrp_{\bSi} \stackrel{\sim}{\to} \ltr_{\bSi}$ is an equivalence of
monoidal dg categories, where the monoidal structure on  $\ltrp_{\bSi}$ is given by the 
tensor product of quasicoherent sheaves, and the monoidal structure on
$\ltr_{\bSi}$ is given by convolution.
\end{corollary}

\section{Intrinsic Characterizations}
\label{sec:reformulate}

In Section \ref{sec:Theta-Theta} above,  we have given an equivalence $\kappa_\bSi$ between 
a certain dg category of quasicoherent sheaves on the toric DM stack $\cX_{\bSi}$---which we have called $\ltrp_\bSi$---and 
a certain category of constructible sheaves on the real vector space $M_\bR$---which we have called $\ltr_\bSi$.  
The categories $\ltr_\bSi$ and $\ltrp_\bSi$ are defined by their set of generating objects 
$\{\Theta(\si,\chi)\}_{(\si, \chi) \in \Ga(\bSi)}$ and $\{\Theta'(\si,\chi)\}_{(\si,\chi) \in\Ga(\bSi)}$.  
In this section we give intrinsic characterizations of these categories.

\subsection{Shard arrangements}
\label{sec:shard}

\begin{definition}[shard arrangement]
Let $\bSi=(N,\Si,\beta)$ be a stacky fan.  A \emph{shard arrangement} for $\bSi$ is a closed set 
$E \subset M_\bR \times N_\bR$ of the form
$$
E = \bigcup_{i = 1}^k (\si_i)_{\chi_i}^\perp \times -\sigma_i
$$
where for each $i = 1, \ldots, k$, $(\sigma_i,\chi_i)\in \Ga(\bSi)$.
\end{definition}

\begin{definition}
A sheaf $F$ on $M_\bR$ is called a \emph{$\bSi$-shard sheaf} if it is cohomologically bounded and constructible, it has finite-dimensional fibers, and its singular support $\SS(F)$ is a subset of a shard arrangement.  Let $\Shard(M_\bR;\bSi)$ denote the triangulated dg category of $\bSi$-shard sheaves on $M_\bR$.
\end{definition}

The union of all possible shard arrangements is a conical Lagrangian $\LbS$ in $M_\bR\times N_\bR=T^*M_\bR$:
\begin{definition}
Define the conical Lagrangian
\begin{equation}
\LbS = \bigcup_{\tau\in\Sigma} \bigcup_{\chi\in M_\tau}
\tau_\chi^\perp \times -\tau \subset M_\bR\times N_\bR = T^*M_\bR.
\end{equation}
\end{definition}

For example, we have $\Sh_{cc}(M_\bR;\LbS) \subset \Shard(M_\bR;\bSi)$.
More generally the sheaves $\Theta(\sigma,\chi)$ belong to $\Shard(M_\bR;\bSi)$
but do not have compact support.
Suppose that $\tau\subset \si \in \Si$, let $f_{\tau\si}$ denote
the inclusion $N_\tau \to N_\si$, and let $f_{\tau\si}^*: M_\sigma\to M_\tau$
be the dual map of $\bar{f}_{\tau\sigma}:\baN_\tau\to \baN_\si$.

\begin{proposition} \label{prop:SStheta}
After identifying $T^*M_\bR$ with $M_\bR \times N_\bR$, the singular support of
$\Theta(\sigma,\chi)$ is given by the following:
$$
\SS(\Theta(\sigma,\chi)) = \bigcup_{\tau \subset \sigma}
(\tau_{\chi_\tau}^\perp \cap \sigma_\chi^\vee) \times -\tau.
$$
where $\chi_\tau = f^*_{\tau\sigma}(\chi)$.
\end{proposition}

\begin{theorem} \label{thm:thetagens} 
The dg category $\Shard(M_\bR;\bSi)$ is
quasi-equivalent to $\ltr_\bSi$. In other words, every shard sheaf is
quasi-isomorphic to a bounded complex of the form
$$
\cdots \to \bigoplus_i \Theta(\si_i,\chi_i)    \to \bigoplus_j
\Theta(\si_j,\chi_j) \to \cdots
$$
where each sum is finite.
\end{theorem}
\begin{proof}
This is an immediate consequence of \cite[Theorem 5.2]{more}.
\end{proof}

Combining Theorem \ref{Theta-updown}, Theorem \ref{thm:thetagens},
and the results in \cite[Section 5]{more}, we obtain:
\begin{corollary}[intrinsic characterization of $\langle\Theta\rangle$]
If $F$ belongs to $\Shard(\tM_\bR,\bSi)$ the
$\hat{\beta}^! F$ belongs to $\Shard(M_\bR,\bSi)$. 
The following square of functors commutes up to natural isomorphism:
$$
\begin{CD}
\langle \Theta\rangle_{\tSi} @>{\cong}>> \Shard(\tM_\bR,\tSi)\\
@V{\hat{\beta}^!}VV   @V{\hat{\beta}^!}VV\\
\langle \Theta\rangle_{\bSi} @>{\cong}>> \Shard(M_\bR,\bSi)
\end{CD}
$$
where all the arrows are quasi-equivalences of dg categories.
\end{corollary}

\subsection{Quasicoherent sheaves with finite fibers}
\label{sec:finfib}

We first recall a definition from \cite[Section 6]{more}.
\begin{definition}
A quasicoherent sheaf (or complex of sheaves) on a scheme $X$ has \emph{finite fibers} if for each closed point $x \in X$ we have
\begin{itemize}
\item $\Tor_i(\cO_x/\mathfrak{m}_x,\cF) := h^{-i}(\cO_x/\frakm_x \stackrel{\mathbf{L}}{\otimes} \cF)$ are finite-dimensional, and
\item $\Tor_i(\cO_x/\frakm_x,\cF) = 0$ for all but finitely many $i \in \bZ$.
\end{itemize}

\end{definition}

If $f:S_1 \to S_2$ is faithfully flat, then we may check whether a quasicoherent sheaf $\cF$ on $S_2$ has finite fibers by showing that $f^* \cF$ on $S_1$ does.  Thus we have a good notion of quasicoherent sheaves with finite fibers on quotient stacks:

\begin{definition}
Let $U$ be a scheme on which a group scheme $G$ acts, and let $\cX=[U/G]$ be the quotient stack.  We say a quasicoherent sheaf (or complex of sheaves) on $\cX$ has \emph{finite fibers}
if the corresponding $G$-equivariant quasicoherent sheaf (or complex of 
quasicoherent sheaves) on $U$ has finite fibers.

Suppose that the $G$-action comes from a group homomorphism $\phi:G\to \tT$
where $\tT$ acts on $U$, and both $G$ and $\tT$ are {\em abelian}.
Then the Picard stack $\cT=[\tT/G]$ acts on $\cX$.
We say a $\cT$-equivariant quasicoherent sheaf on $\cX$ has finite fibers if the corresponding $\tT$-equivariant quasicoherent sheaf on $U$ has finite fibers.
\end{definition}

\begin{remarks}
\label{rmk:finfib}
\begin{enumerate}
\item Let $\cX=[U/G]$ be a toric DM stack. Then any
coherent sheaf on $\cX$ has finite fibers. In particular,
all vector bundles and perfect complexes have finite fibers.

\item It follows from the adjunction formula
$$\dghom_{\cO_x/\frakm_x}(\cF \otimes \cO_x/\frakm_x, \cO_x/\frakm_x) \cong \dghom_X(\cF,\cO_x/\frakm_x)$$
that $\cF$ has finite fibers if and only if $\Ext^i(\cF,\cO_x/\frakm_x)$ is finite-dimensional for all $i$ and vanishes for almost all $i$.
\end{enumerate}
\end{remarks}

Recall that there is a quasi-equivalence of dg categories (see Theorem \ref{thm:mainfirst})
$$
q: \langle \Theta'\rangle_{\tSi} \stackrel{\cong}{\to} \langle \Theta'\rangle_{\bSi}.
$$
where $\langle\Theta'\rangle_{\bSi}$ (resp.
$\langle\Theta'\rangle_\tSi$) is a dg subcategory of
$\cQ^\finfib_\cT(\cX)$ (resp. $\cQ^\finfib_{\tT}(U)$).
It follows from the definition that
$$
q:\cQ^\finfib_{\tT}(U)\stackrel{\cong}{\to} \cQ^\finfib_\cT(\cX).
$$
By \cite[Theorem 6.3]{more},
the  $\langle  \Theta'\rangle_{\tSi}$ is quasi-equivalent
to $\cQ^\finfib_{\tT}(U)$.  We conclude that
$\langle\Theta'\rangle_{\bSi}$ is quasi-equivalent
to $\cQ^\finfib_\cT(\cX)$:

\begin{theorem}[intrinsic characterization of $\langle \Theta'\rangle$]
The following square of functors commutes up to natural isomorphism:
$$
\begin{CD}
\langle\Theta'\rangle_\tSi @>{\cong}>>\cQ^\finfib_{\tT}(U)\\
@V{q}VV @V{q}VV\\
\langle\Theta'\rangle_\bSi @>{\cong}>>\cQ^\finfib_\cT(\cX)
\end{CD}
$$
where all the arrows are quasi-equivalence of dg categories.
\end{theorem}

\subsection{Finite fibers and shard arrangements}

After the main results of Section \ref{sec:shard} and Section \ref{sec:finfib}, and Theorem \ref{thm:mainfirst}, we have:

\begin{theorem}
Let $\cX_\bSi$ be a toric DM stack defined by a stacky fan $\bSi=(N,\Si,\beta)$, 
and let $\tSi$ and $U$ be defined as in Section \ref{sec:toricDM}.
Then the following square of functors commutes up to natural isomorphism:
$$
\begin{CD}
\cQ_{\tT}^\finfib(U) @>{\kappa_\tSi}>> \Shard(\tM_\bR,\tSi)\\
@V{q}VV @VV{\hat{\beta}^!}V\\
\cQ_\cT^\finfib(\cX_\bSi) @>{\kappa_\bSi}>>  \Shard(M_\bR, \bSi).
\end{CD}
$$
where all the arrows are quasi-equivalence of monoidal dg  categories.
\end{theorem}

\section{Perfect Complexes  and Compactly Supported Constructible Sheaves}
\label{sec:Perf-cc}

In this section, $\chi=\chi_\bSi$ is a {\em complete} DM stack,
i.e., it is defined by a stacky fan $\bSi =(N,\Si,\beta)$ where
$\Si$ is a complete fan in $N_\bR$. The coarse
moduli space $X=X_\Si$ of $\cX$ is a complete simplicial toric variety.

\subsection{Generating sets of line bundles}
\label{sec:gensetlinbun}

Let $U\subset \bC^r$, $G_\bSi$ be defined as in Section \ref{sec:toricDM},
so that $\cX=[U/G_\bSi]$. Let $L_\vc$, $L_\vc'$, and $\cL_\vc$
be defined as in Section \ref{sec:Lc}.

\begin{proposition}\label{resolve}
Let $\cF$ be a $\cT$-equivariant coherent sheaf on $\cX$. Then
there is
a $\cT$-equivariant free resolution:
$$
0\to \cV_l\to \cdots \to \cV_0 \to \cF\to 0
$$
where each $\cV_i$ is the direct sum of $\cT$-equivariant line bundles
$\{ \cL_{\vc}\mid \vc\in \bZ^r \}$.
\end{proposition}

\begin{proof} This can be proved by a slight modification of the
proof of \cite[Theorem 4.6]{BH}. We outline the argument here
and refer to \cite[Section 4]{BH} for details. 
  
Let $\tcF$ be the $\tT$-equivariant coherent sheaf on $U$ which
descends to the $\cT$-equivariant coherent sheaf $\cF$ on $\cX=\cX_{\bSi}$.  
It suffices to show that there exists a $\tT$-equivariant free resolution
\begin{equation}\label{eqn:resolve-tcF}
0\to \tcV_l \to \cdots \to \tcV_0\to \tcF\to 0
\end{equation}
where each $\tcV_i$ is the direct sum of $\tT$-equivariant
line bundles $\{ \tL_{\vc}\mid \vc\in \bZ^r\}$.

Let $S=H^0(U,\tcF)$, and let $A=\bC[z_1,\ldots,z_r]$. By 
\cite[Lemma 4.7]{BH}, $S$ is a finitely generated $A$-module. 
Let $\tcF'$ be the coherent sheaf on $\bC^r=\Spec A$ determined
by the finitely generated $A$-module $S$. Then
$\tcF'\bigr|_U =\tcF$.

The $\tT$-linearization on $\tcF$ gives rise to a $\tT$-action on $S$, which
is compatible with the $\tT$-action on $A$.  Therefore $A$ and $S$ are
graded by $\tM=\Hom(\tT,\bC^*)$. There is a surjection
$$
F_0\to S\to 0
$$
where $F_0$ is a direct sum of rank one $A$-modules generated by
eigenelements of $\tT$. Let $I\subset A$ be the maximal
idea generated by $z_1,\ldots, z_r$. We may choose $F_0\to S$ such 
that the surjective map $F_0/IF_0\to S/IS$ of $\bC$-vector spaces
is an isomophism. Then the kernel $S_1$ of $F_0\to S$ is contained
in $IF_0$. We replace $S$ by $S_1$ and repeat the procedure, and we
obtain an exact sequence of $\tT$-equivariant $A$-modules
$$
F_1\stackrel{\phi_1}{\to} F_0\to S\to 0,
$$
where $F_1$ is a direct sum of rank one $A$-modules generated
by eigenelements of $\tT$. We continue this procedure and obtain 
a $\tT$-equivariant, free resolution
$$
 \cdots \to F_i\stackrel{\phi_i}{\to} F_{i-1} \to \cdots \to F_1\stackrel{\phi_1}{\to} F_0\to S\to 0, 
$$
where the image of each $\phi_i:F_i\to F_{i-1}$ is contained in $IF_{i-1}$.
The above resolution is a minimal graded resolution of the finitely generated
graded $A$-module $S$, so it has finite length $l\leq r$ \cite[Chapter 19]{Es}.
Therefore we have a finite, $\tT$-equivriant free resolution
\begin{equation}\label{eqn:resolve-M}
0\to F_l\stackrel{\phi_l}{\to} \cdots \stackrel{\phi_2}{\to} F_1\stackrel{\phi_1}{\to }  F_0\to M\to 0,
\end{equation}
where each $F_i$ is a direct sum of rank one $A$-modules generated by
eigenelements of $\tT$. Eigenelements of $\tT$ are of the form
$z_1^{-c_1}\cdots z_r^{-c_r}$, $c_1,\ldots, c_r\in \bZ$.
The resolution \eqref{eqn:resolve-M} defines a finite, $\tT$-equivariant, free resolution
\begin{equation}\label{eqn:resolve-tcFprime}
0\to \tcV'_l\stackrel{\phi_l}{\to} \cdots \stackrel{\phi_2}{\to} \tcV'_1\stackrel{\phi_1}{\to }  \tcV'_0\to \tcF'\to 0,
\end{equation}
where each $\tcV'_i$ is a direct sum of $\tT$-equivariant line bundles $\{\tL'_{\vc}\mid \vc\in \bZ^r\}$.
Restricting \eqref{eqn:resolve-tcFprime} to $U$, we obtain
a finite, $\tT$-equivairant, free resolution of the desired form \eqref{eqn:resolve-tcF}.
\end{proof}

\begin{corollary}\label{K} 
Let $\cX$ be a toric DM stack.  
\begin{enumerate}
\item[(a)] The cohomology category of $\Perf_\cT(\cX)$ is $D_{\cT}(\cX)$, the
bounded derived category of $\cT$-equivariant coherent sheaves on $\cX$.
\item[(b)] Any element in $K_\cT(\cX)= K(\Perf_\cT(\cX))$ can be written as
a finite sum $a_1 \cL_{\vc_1}+\cdots + a_N \cL_{\vc_N}$, where $a_1,\ldots, a_N \in \bZ$.
\end{enumerate}
\end{corollary}

\begin{theorem}\label{Lc-generate}
Let $\cX$ be a toric DM stack. Then $\Perf_\cT(\cX)$ is generated by $\cT$-equivariant line bundles.
\end{theorem}
\begin{proof}
Let $\cA$ be the full triangulated dg subcategory of $\Perf_\cT(\cX)$
generated by $\cT$-equivariant line bundles. We need to
show that $\cA=\Perf_\cT(\cX)$. By Proposition \ref{resolve},  $\cA$
is a full, dense triangulated subcategory of $\Perf_\cT(\cX)$. (Recall
that a triangulated subcategory is called {\em dense} if every
object is a direct summand of an object in the subcategory.)
By \cite[Theorem 2.1]{Th}, to show that $\cA=\Perf_\cT(\cX)$ it suffices
to show that the subgroup $K(\cA)$ of $K(\Perf_\cT(X)) = K_\cT(\cX)$
is equal to $K_\cT(\cX)$, which follows from Corollary \ref{K} (b).
\end{proof}

\subsection{Perfect complexes and compact support}
\label{sec:perfsupp}
In this section, we prove that for any complete toric DM stack 
$\cX_\bSi$ there is a quasi-equivalence of monoidal dg categories
$$
\Perf_\cT(\cX_{\bSi}) \cong \Sh_{cc}(M_\bR;\LbS).
$$
As in \cite[Section 7]{more},
the proof  makes use of the monoidal structure, in particular the fact that a complex of quasicoherent sheaves is
perfect if and only if it is \emph{dualizable}.
The argument in the proof of \cite[Theorem 7.4]{more} shows the following.

\begin{theorem} \label{thm:duality}
Suppose $\cX$ is a complete toric DM stack, and let
$\kappa_{\bSi}:\Perf_\cT(\cX) \hookrightarrow \Sh_c(M_\bR;\LbS)$ 
denote the functor defined in Section \ref{sec:ccc-line}.
There is a natural isomorphism
$$
\kappa_{\bSi}(\cF^\vee) \cong -\cD(\kappa_{\bSi}(\cF)).
$$
\end{theorem}

\begin{lemma}[toric Chow's lemma for toric DM stacks]\label{chow}
Let $\cX_{\bSi}$ be a complete toric DM stack.
Then there exists a stacky fan $\bSi'=(N,\Si', \beta')$ and
a morphism of stacky fans $f:\bSi'\to \bSi$ (equivalently, a morphism
of  toric DM stacks $u:\cX_{\bSi'}\to\cX_{\bSi}$), where
\begin{enumerate}
\item[(i)] the toric variety $X_{\Si'}$ is projective,
\item[(ii)] the morphism $X_{\Si'}\to X_\Si$ is birational.
\end{enumerate}
\end{lemma}
\begin{proof}
By toric Chow's lemma (cf.\cite{Da}, \cite[Proposition 2.17]{Od}), there exists a refinement
$\Si'$ of $\Si$ which induces a morphism
$X_{\Si'}\to X_\Si$ of toric varieties such that (i) and (ii) hold.
$\cX_{\bSi'}$ is the toric DM stack  obtained by taking the fiber
product $X_{\Si'}\times_{X_\Si}\cX_{\bSi}$.
\end{proof}

\begin{theorem}[coherent-constructible correspondence for toric DM stacks]\label{thm:main3}
Suppose that $\cX=\cX_{\bSi}$ is a complete toric DM stack
defined by a stacky fan $\bSi=(N,\Si, \beta)$.
Then $\kappa_{\bSi}$ restricts to a quasi-equivalence of  monoidal dg 
categories $\Perf_\cT(\cX) \cong \Sh_{cc}(M_\bR;\LbS)$.
\end{theorem}
\begin{proof} 
We first show that $\kappa_{\bSi}$ carries $\Perf_\cT(\cX)$ into $\Sh_{cc}(M_\bR;\LbS)$.
By Theorem \ref{Lc-generate}, it suffices to show that
if $\cL$ is a $\cT$-equivariant line bundle on $\cX$, then $\kappa(\cL)$
has compact support.  By functoriality, and by Lemma \ref{chow}, we may assume that the coarse moduli space $X_\Si$  is projective.  Then there exist
$\cT$-equivariant, $\bQ$-ample line bundles $\cL_1, \cL_2$ on $\cX$ such that
$\cL\cong \cL_1\otimes \cL_2^{-1}$. Then 
$$
\kappa_\bSi(\cL)=\kappa_{\bSi}(\cL_1)\star\kappa_{\bSi}(\cL_2^{-1}),
$$
where $\kappa_{\bSi}(\cL_2^{-1})=-\cD(\kappa_{\bSi}(\cL_2))$
by Theorem \ref{thm:duality}. By Theroem \ref{thm:algebraic},
$\kappa_{\bSi}(\cL_1)$ and $\kappa_{\bSi}(\cL_2)$ have compact supports.
Therefore $\kappa_{\bSi}(\cL)$ has compact support.

It is clear from our earlier results that $\kappa_{\bSi}:\Perf_\cT(\cX_{\bSi}) \to \Sh_{cc}(M_\bR;\LbS)$
is a fully faithful embedding of monoidal dg categories.
To complete the proof of the theorem it remains to show that
$\kappa_{\bSi}$ is essentially surjective. Suppose that 
$F\in Sh_{cc}(M_\bR;\LbS)$. Then $F\in \Shard(M_\bR;\bSi)$, so there exists
$\cG \in \langle \Theta'\rangle_{\bSi}$ such that 
$\kappa_\bSi(\cG)\cong F$. We also have $\cD F\in \Shard(M_\bR;\bSi)$, 
so there exists $\cH\in \langle \Theta'\rangle_{\bSi}$
such that $\kappa_\bSi(\cH)\cong -\cD F$. 
By \cite[Lemma 7.5]{more}, $F\to F\star (-\cD F) \star F\to F$
is the identity map. Therefore, 
$\cG \to \cG \otimes \cH \otimes \cG \to \cG$ is the identity map.
So $\cG$ is strongly dualizable, thus perfect (cf. \cite[Proposition 7.3]{more}).
\end{proof}

\section{Equivariant HMS for Toric DM  Stacks} \label{sec:HMS}

Recall that $Fuk(T^*M_\bR; \LbS)$ is a subcategory in the
Fukaya category $Fuk(T^*M_\bR)$, consisting of Lagrangian branes $L$
whose boundary at infinity $L^\infty$ is a subset of the infinity
boundary of $\Lambda_\bSi$. We use $F(T^*M_\bR; \Lambda_\bSi)$ to
denote the $A_\infty$ triangulated envelope of $Fuk(T^*M_\bR;
\LbS)$. The following theorem is a direct
consequence of Theorem \ref{thm:main3}, Theorem \ref{thm:algebraic} and the
microlocalization functor in \cite{N, NZ}.

\begin{theorem} \label{thm:main4}
If $\cX_{\bSi}$ is a complete DM stack then there is
a quasi-equivalence of triangulated $A_\infty$-categories:
$$
\tau: \Perf_\cT(\cX_{\bSi})\stackrel{\sim}{\to} F(T^*M_\bR; \LbS)
$$
which is given by the composition
$$
\Perf_\cT(\cX_{\bSi})\stackrel{\kappa_\bSi}{\to} Sh_{cc}(M_\bR;\LbS)
\stackrel{\mu_{M_\bR} }{\to}F(T^*M_\bR;\LbS).
$$
If $\cO_{\cX_\Si}(\uchi)$ is a $\cT$-equivariant $\bQ$-ample
line bundle associcated to a twisted polytope $\uchi=(\chi_1,\ldots,\chi_r)$
of $\bSi$ then $\tau(\cO_{\cX_\Si}(\uchi))$ is a costandard brane over
the interior of the convex hull $\mathbb{P}$ of $\chi_1,\ldots,\chi_r\in M_\bR$.
\end{theorem}

We also have functoriality involving Fukaya categories \cite{N}. 
Let $Y_0$ and $Y_1$ be real analytical manifolds. An
object $\cK$ in $\Sh_c(Y_0\times Y_1)$ defines a functor
\begin{equation}
\Phi_{\cK !}: \Sh_c(Y_0)\to \Sh_c(Y_1),\quad \cF \mapsto
p_{1!}(\cK\otimes p_0^* \cF),
\end{equation}
where $p_0, p_1$ are projections of $Y_0\times Y_1$ to the
corresponding components. For a Lagrangian brane $L$ in
$Fuk(T^*(Y_0\times Y_1))$, define a functor
$$\Psi_{L!}:= \mu_{Y_1}\circ \Phi_{\cK!} \circ \mu_{Y_0}^{-1}:
F(T^*Y_0)\to F(T^*Y_1).$$

For two toric DM stacks  $\cX_1=\cX_{\bSi_1}$ and
$\cX_2=\cX_{\bSi_2}$ and a morphism of stacky fans
$f: \bSi_1= (N_1, \Si_1, \beta_1)\to \bSi_2 =(N_1,\Si_2,\beta_2)$, 
let $v:M_{2,\bR}\to M_{1,\bR}$  and $u:\cX_1 \to \cX_2$ be defined
as in Section \ref{sec:ccfunctoriality}. Set the Lagrangian brane
$L_v$ to be the conormal bundle $T^*_{\Gamma_v}(M_{2,\bR}\times
M_{1,\bR})$, where $\Gamma_v$ is the graph of $v$ in
$M_{2,\bR}\times M_{1,\bR}$. The derivation of the following
theorem is the same as that of \cite[Theorem 3.7]{hms}.

\begin{theorem}[functoriality]\label{thm:Fuk-functoriality}
For two complete toric DM stacks $\cX_1=\cX_{\bSi_1}$ and
$\cX_2=\cX_{\bSi_2}$  and a morphism of stacky fans
$$
f: \bSi_1= (N_1,\Si_1, \beta_1)\to
\bSi_2 =(N_2,\Si_2,\beta_2)
$$
and associated maps 
$u: \cX_1\to \cX_2$, $u|_{\cT_1}:\cT_1\to \cT_2$, $v: M_{2,\bR}\to M_{1,\bR}$, the following diagram
commutes up to quasi-isomorphism:
$$
\xymatrix{
 \Perf_{\cT_2}(\cX_2) \ar[d]_{u^*}  \ar[r]^{\kappa_2}  &
Sh_{cc}(M_{2,\bR}; \Lambda_{\bSi_2}) \ar[r]^{ \mu_{M_{2,\bR}}  }  \ar[d]^{v_!}&
F(T^*M_{2,\bR};\Lambda_{\bSi_2})\ar[d]^{\Psi_{L_v!} } \\
\Perf_{\cT_1}(\cX_1) \ar[r]^{\kappa_1}  &
 Sh_{cc}(M_{1,\bR}; \Lambda_{\bSi_1}) \ar[r]^{\mu_{M_{1,\bR}}} &
F(T^*M_{1,\bR};\Lambda_{\bSi_1})
}
$$
where $\kappa_i=\kappa_{\bSi_i}$.
\end{theorem}

In \cite[Section 3.4]{hms}, we define a product structure
on the Fukaya category $F(T^*M_\bR)$ by 
\begin{equation}\label{eqn:diamond}
L_1 \diamond L_2= \Psi_{L_v!}(L_1\times L_2)
\end{equation} 
where $v:M_\bR \times M_\bR$ is the addition map. The following
is a special case of \cite[Proposition 3.9]{hms}: 
\begin{proposition} (the microlocalization interwines the product structures)
The microlocalization functor 
$\mu_{M_\bR}: Sh_{cc}(M_\bR)\stackrel{\sim}{\to} F(T^*M_\bR)$ 
intertwines the monoidal product on  $Sh_{cc}(M_\bR)$ given by the convolution, 
and the product structure on $F(T^*M_\bR)$ given by the product $\diamond$ 
defined by \eqref{eqn:diamond}, up to a quasi-isomorphism: 
the functors $\mu_{M_\bR}(-\star-)$ and $\mu_{M_\bR}(-)\diamond \mu_{M_\bR}(-)$ 
are quasi-isomorphic in the category of $A_\infty$-functors from 
$Sh_{cc}(M_\bR)\times Sh_{cc}(M_\bR)$ to $F(T^*M_\bR)$.
\end{proposition}

\begin{corollary}\label{cor:tensor2} 
The quasi-equivalence $\tau: \Perf_\cT(\cX_{\bSi})\stackrel{\sim}{\to} F(T^*M_\bR;\LbS)$ 
interwines the monoidal product on $\Perf_\cT(\cX_\bSi)$ given by the tensor product
$\otimes$ of sheaves,  and the product structure on $Fuk(T^*M_\bR;\LbS)$ given by the product $\diamond$ 
defined by \eqref{eqn:diamond}, up to a quasi-isomorphism: 
the functors $\tau(-\otimes-)$ and $\tau(-)\diamond \tau(-)$ 
are quasi-isomorphic in the category of $A_\infty$-functors from 
$\Perf_\cT(\cX_{\bSi})\times \Perf_\cT(\cX_{\bSi})$ to $Fuk(T^*M_\bR;\LbS)$.
\end{corollary}

\begin{theorem}[equivariant homological mirror symmetry for toric DM stacks]
Let $\cX_\bSi$ be a complete toric DM stack defined by a stack fan $\bSi$. Then
there is an equivalence of tensor triangulated categories:
$$
H(\tau): D_\cT(\cX_{\bSi}) \stackrel{\sim}{\to} DF(T^*M_\bR;\LbS).
$$
\end{theorem}

\end{document}